\documentclass{article}

\usepackage{amsmath}
\usepackage{amssymb}
\usepackage{amsthm}
\usepackage{enumerate}
\usepackage{amscd}
\usepackage{url}
\usepackage[T1]{fontenc}
\usepackage[english,francais]{babel}

\theoremstyle{plain}
\newtheorem{probleme}{Probl\`eme}

\newtheorem{theoreme}{Th\'eor\`eme}
\newtheorem{lemme}[theoreme]{Lemme}
\newtheorem{proposition}[theoreme]{Proposition}
\newtheorem{corollaire}[theoreme]{Corollaire}
\newtheorem{definition}[theoreme]{D\'efinition}

\theoremstyle{definition}
\newtheorem{remarque}[theoreme]{Remarque}

\usepackage{url}

\newcommand{\beq}{\begin{equation}}
\newcommand{\eeq}{\end{equation}}

\DeclareMathOperator{\Div}{Div}
\DeclareMathOperator{\Supp}{Supp}
\DeclareMathOperator{\pgcd}{pgcd}
\DeclareMathOperator{\ev}{ev}
\DeclareMathOperator{\im}{im}
\DeclareMathOperator{\tr}{tr}

\renewcommand{\div}{\textrm{div}\,}

\newcommand{\PP}{{\mathbb{P}}}
\newcommand{\Z}{\mathbb{Z}}
\newcommand{\N}{\mathbb{N}}
\newcommand{\Q}{\mathbb{Q}}
\newcommand{\R}{\mathbb{R}}
\newcommand{\F}{\mathbb{F}}
\newcommand{\ie}{\emph{i.e. }}
\newcommand{\cad}{c'est-\`a-dire }
\newcommand{\longto}{\longrightarrow}
\newcommand{\tens}{\otimes}
\newcommand{\cL}{\mathcal{L}}
\newcommand{\cO}{\mathcal{O}}
\newcommand{\cH}{\mathcal{H}}
\newcommand{\cS}{\mathcal{S}}

\newcommand{\un}{\mathbf{1}}

\newcommand{\equivaut}{\Longleftrightarrow}

\begin{document}
\selectlanguage{francais}


\title{Diviseurs de la forme $2D-G$ sans sections\\ et rang de la multiplication dans les corps finis}


\author{Hugues Randriam}

\maketitle

\begin{abstract}
Soient $X$ une courbe alg\'ebrique, d\'efinie sur un corps parfait,
et $G$ un diviseur sur $X$.
Si $X$ a suffisamment de points,
on montre comment construire un diviseur $D$ sur $X$
tel que $l(2D-G)=0$, pour toute valeur de $\deg D$ telle que
ceci soit compatible avec le th\'eor\`eme de Riemann-Roch
(\ie jusqu'\`a $\deg D=\left\lfloor\frac{1}{2}(g(X)-1+\deg G)\right\rfloor$).
On g\'en\'eralise aussi cette construction au cas d'un nombre fini de
contraintes
$l(k_iD-G_i)=0$, o\`u $|k_i|\leq2$.

Un r\'esultat de ce type avait \'et\'e \'enonc\'e par
Shparlinski-Tsfasman-Vladut, en relation avec la m\'ethode
de Chudnovsky-Chudnovsky pour estimer la complexit\'e bilin\'eaire de la
multiplication dans les corps finis par interpolation
sur les courbes~; malheureusement leur preuve
\'etait erron\'ee.
Ainsi notre travail permet de corriger la preuve de
Shparlinski-Tsfasman-Vladut et montre que leur estimation
$m_q\leq 2\left(1+\frac{1}{A(q)-1}\right)$
est valable, du moins d\`es lors que $A(q)\geq 5-\frac{14q^2-4}{q^4+2q^2-1}$. 
On corrige aussi un \'enonc\'e de Ballet qui souffre du m\^eme
probl\`eme, et on indique enfin quelques autres applications possibles.
\end{abstract}

\selectlanguage{english}
\begin{abstract}
Let $X$ be an algebraic curve, defined over a perfect field,
and $G$ a divisor on $X$.
If $X$ has sufficiently many points,
we show how to construct a divisor $D$ on $X$ such that
$l(2D-G)=0$,
for any value of $\deg D$ such that this is compatible with
the Riemann-Roch theorem
(that is, up to $\deg D=\left\lfloor\frac{1}{2}(g(X)-1+\deg G)\right\rfloor$).
We also generalize this construction to the case of a finite
number of constraints,
$l(k_iD-G_i)=0$, where $|k_i|\leq2$.

Such a result was previously claimed by Shparlinski-Tsfasman-Vladut,
in relation with the Chudnovsky-Chudnovsky method for estimating
the bilinear complexity of the multiplication in finite fields based
on interpolation on curves;
unfortunately, their proof was flawed.
So our work fixes the proof of Shparlinski-Tsfasman-Vladut and shows
that their estimate
$m_q\leq 2\left(1+\frac{1}{A(q)-1}\right)$
holds, at least when $A(q)\geq 5-\frac{14q^2-4}{q^4+2q^2-1}$.
We also fix a statement of Ballet that suffers from the same problem,
and then we point out a few other possible applications.
\end{abstract}
\selectlanguage{francais}

\section*{Introduction}

Plusieurs questions de math\'ematiques discr\`etes
et d'informatique th\'eorique se ram\`enent au probl\`eme suivant~:

\begin{probleme}
\label{pb1}
Soient $X$ une courbe alg\'ebrique sur un corps $K$,
et $r\geq 1$ un entier naturel.
\'Etant donn\'es des entiers relatifs $k_1,\dots,k_r\in\Z$
et des diviseurs $K$-rationnels $G_1,\dots,G_r$ sur $X$, construire
un diviseur $K$-rationnel $D$ sur $X$,
tel que les diviseurs $k_1D-G_1,\dots,k_rD-G_r$ soient sans sections~:
\beq
l(k_1D-G_1)=\cdots=l(k_rD-G_r)=0.
\eeq
\end{probleme}

(Ici, ``courbe'' signifiera toujours~: courbe projective lisse
g\'eom\'etriquement irr\'eductible.)

Citons quelques questions rentrant dans ce cadre
(on supposera en g\'en\'eral que $K$
est un corps fini $\F_q$)~:

\begin{itemize}
\item la construction de codes lin\'eaires intersectants ou,
plus g\'en\'eralement, de codes s\'eparants ou ``frameproof'' (\cite{Xing2002})
\item l'estimation de la complexit\'e bilin\'eaire de la multiplication
dans les corps finis, par interpolation sur les courbes (\cite{ChCh})
\item la construction de syst\`emes de partage de secret lin\'eaires 
avec propri\'et\'e de multiplication (\cite{ChenCramer}).
\end{itemize}

En fait, ces sujets ne sont pas compl\`etement ind\'ependants~:
les liens entre codes intersectants et complexit\'e bilin\'eaire sont
connus de longue date (\cite{LW}), et par ailleurs, l'existence
d'algorithmes de multiplication \`a faible complexit\'e permet parfois,
par des arguments de descente
du corps de base, d'am\'eliorer certaines constructions
de syst\`emes de partage de secret (\cite{CCCX}).

\medskip

Suivant l'approche de \cite{ITW2010}, elle-m\^eme inspir\'ee
de \cite{Xing2002}, le probl\`eme~\ref{pb1} peut se reformuler 
comme suit (voir aussi \cite{Cascudo}).

Notons $g$ le genre de $X$, et
$J$ sa jacobienne.
Supposant que $X$ admette au moins un point $K$-rationnel,
notons $j:X\longto J$ le plongement d'Abel-Jacobi associ\'e,
$W_1=j(X)$ son image, et plus g\'en\'eralement
$0=W_0\subset W_1\subset W_2\subset\cdots\subset W_{g-1}=\Theta\subset W_g=J$
les sommes it\'er\'ees de $W_1$ avec elle-m\^eme, ou de fa\c{c}on
\'equivalente, les images des applications d'Abel-Jacobi sup\'erieures
(et $W_i=\emptyset$ si $i<0$).
Notons aussi $\cH=J(K)$.

\begin{probleme}
\label{pb2}
Avec les notations pr\'ec\'edentes, \'etant donn\'es un entier 
$r\geq 1$, des entiers relatifs $k_1,\dots,k_r$ et
$n_1,\dots,n_r\in\Z$,
et des \'el\'ements $\kappa_1,\dots,\kappa_r\in\cH$,
trouver un entier $d$
tel que
$k_1^{-1}(W_{k_1d-n_1}+\kappa_1)$, $\dots$, $k_r^{-1}(W_{k_rd-n_r}+\kappa_r)$
ne recouvrent pas $\cH$.
\end{probleme}

On passe du probl\`eme~\ref{pb1} au probl\`eme~\ref{pb2} en posant
$n_i=\deg G_i$ et $\kappa_i=j(G_i)$
pour $1\leq i\leq r$, et $d=\deg D$.

Souvent, on demandera que cette construction optimise une certaine
fonction de $d$.
Par exemple, 
le th\'eor\`eme de Riemann-Roch (ou le fait que $W_g=J$)
impose que, si un tel $d$ existe,
il v\'erifie n\'ecessairement
\beq
\label{max<=g-1}
\max(k_1d-n_1,\dots,k_rd-n_r)\leq g-1.
\eeq
On pourra vouloir essayer de maximiser cette derni\`ere quantit\'e.



C'est ce qui est fait dans ce travail~:
on montre que, sous de bonnes conditions,
la borne \eqref{max<=g-1} peut essentiellement \^etre atteinte~;
on aura notamment besoin pour cela de supposer
que $X$ a suffisamment de points.

\medskip

Indiquons maintenant quelques m\'ethodes utilis\'ees pour
attaquer les probl\`emes~\ref{pb1} et~\ref{pb2}.
\begin{enumerate}[(i)]
\item Argument de degr\'e~: si $d$ v\'erifie
$\max(k_1d-n_1,\dots,k_rd-n_r)<0$,
alors $d$ est trivialement solution du probl\`eme~\ref{pb2},
et mieux, n'importe quel diviseur $D$ de degr\'e $d$ est solution
du probl\`eme~\ref{pb1}.
\item Argument de cardinalit\'e~: pour que
les $k_i^{-1}(W_{k_id-n_i}+\kappa_i)$ ne recouvrent pas $\cH$, 
il suffit que leur r\'eunion (ou plus pr\'ecis\'ement,
celle de leurs points rationnels)
soit de cardinalit\'e strictement plus petite.
Il s'agit donc de pouvoir majorer les $|k_i^{-1}(W_{k_id-n_i}+\kappa_i)(K)|$.

Remarquons que le point (i) en est un cas particulier,
car si $k_id-n_i<0$, alors $W_{k_id-n_i}$ est vide. 
Cet argument est utilis\'e notamment dans \cite{ChCh}.

En g\'en\'eral on peut utiliser la majoration
$|k^{-1}(W)(K)|\leq|\cH[k]||W(K)|$, valable pour toute
sous-vari\'et\'e $W$ de $J$ (o\`u $\cH[k]$ est le sous-groupe
de $k$-torsion de $\cH$),
ce qui donne une condition suffisante sur $d$~:
\beq
\label{H[k]W<h}
|\cH[k_1]||W_{k_1d-n_1}(K)|+\cdots+|\cH[k_r]||W_{k_rd-n_r}(K)|<|\cH|.
\eeq
Une telle approche est utilis\'ee dans \cite{STV},
mais avec une erreur~: les auteurs ont oubli\'e la contribution de la torsion. 
Elle est utilis\'ee aussi, de fa\c{c}on correcte, dans \cite{Xing2002},
dans le cas d'une seule contrainte~; et enfin le cas g\'en\'eral
est trait\'e, sous une forme essentiellement identique \`a celle
pr\'esent\'ee ici, dans \cite{Cascudo}
(o\`u l'erreur de \cite{STV} est signal\'ee).

Pour exploiter \eqref{H[k]W<h}, il faut pouvoir estimer
les $|W_n(K)|$ et les $|\cH[k]|$.
Pour le second point, on dispose de la majoration classique
et valable en toute g\'en\'eralit\'e $|\cH[k]|\leq k^{2g}$
(utilis\'ee par exemple dans \cite{Xing2002})~;
mais quitte \`a choisir convenablement la courbe $X$, on peut
vouloir esp\'erer faire mieux.
On s'int\'eressera notamment au cas asymptotique,
o\`u le genre $g$ tend vers l'infini, et o\`u le nombre
de points de $X$ cro{\^\i}t lin\'eairement avec $g$.
On cherche donc des courbes ayant ``beaucoup'' de points
mais dont le groupe de classes ait ``peu'' de $k$-torsion.
Il est naturel d'introduire une constante mesurant la croissance
de cette $k$-torsion par rapport \`a $g$ dans une telle famille
de courbes,
et ceci a \'et\'e fait, ind\'ependamment~:
\begin{itemize}
\item dans \cite{Cascudo}, avec la ``limite de torsion'' $J_k(q,a)$
\item dans \cite{ITW2010}, avec la constante $\delta_k^-(q)$.
\end{itemize}
Ces deux quantit\'es sont reli\'ees par la formule
$J_k(q,A(q))=\delta_k^-(q)\log_qk$.
Dans \cite{ITW2010} on trouve des arguments heuristiques
en faveur de la conjecture
$\delta_k^-(q)=0$.

\item Construction explicite~: il s'agit de construire effectivement
un diviseur $D$ solution du probl\`eme~\ref{pb1} (par contraste avec
la m\'ethode (ii), qui est non-constructive).
Ceci est fait, au moins pour l'application aux codes intersectants, dans \cite{21sep},
et on se propose ici de g\'en\'eraliser et d'am\'eliorer encore ces r\'esultats.
\end{enumerate} 

Notre construction repose sur l'hypoth\`ese que $X$ a suffisamment
de points~; essentiellement, on demande une in\'egalit\'e
du type $\frac{|X(K)|}{g}\geq C(k_1,\dots,k_r)$, pour une constante
$C(k_1,\dots,k_r)$ qu'on essaiera d'estimer aussi pr\'ecis\'ement
que possible, du moins lorsque tous les $k_i$ sont de valeur absolue
au plus $2$ (ce qui suffit pour les applications mentionn\'ees ici).

Si l'on veut utiliser cette m\'ethode dans un cadre asymptotique,
avec $g$ tendant vers l'infini,
on voit qu'il faut supposer que $r$ est fix\'e.
C'est le cas pour les applications aux codes intersectants ($r=1$)
ou \`a la complexit\'e bilin\'eaire de la multiplication ($r=2$).
Il se trouve que
la construction fournit alors un diviseur $D$ dont le degr\'e est
essentiellement celui qu'aurait donn\'e la m\'ethode (ii) sous
l'hypoth\`ese que les $\delta_{k_i}^-(q)$ sont tous nuls~;
et le th\'eor\`eme de Riemann-Roch,
sous la forme d'une version asymptotique de l'in\'egalit\'e \eqref{max<=g-1},
implique qu'on ne peut pas faire mieux.

En revanche, pour les applications au partage de secret,
en g\'en\'eral $r$ cro{\^\i}t avec $g$, ce qui semble interdire
l'usage de notre m\'ethode, du moins sous ses formes les plus na{\"\i}ves.



\section{Formules de Pl\"ucker}
\label{section1}

On donne une variante des formules de Pl\"ucker d'ordre $1$ et $2$
adapt\'ee \`a nos besoins
(on appellera ici formule de Pl\"ucker toute estimation sur le
nombre de points en lesquels la suite des sauts de Weierstrass
d'un diviseur, jusqu'\`a un ordre donn\'e, diff\`ere de son comportement
g\'en\'erique~;
pour une version plus g\'en\'erale, voir par exemple \cite{Laksov}).

Ce premier lemme correspond \`a l'ordre $1$~:

\begin{lemme}
\label{lemme1}
Soient $X$ une courbe de genre $g$ sur un corps $K$,
et $\cS\subset X(K)$ un ensemble de points de $X$.
\begin{enumerate}[a)]
\item
Soit $A$ un diviseur $K$-rationnel sur $X$ tel que
\beq
i(A)=l(A)-(\deg A+1-g)\geq 1.
\eeq
Supposons que pour tout $P\in\cS$ on ait $l(A+P)>l(A)$.
Alors
\beq
|\cS|\leq g-l(A).
\eeq
(Si $\deg A=-1$ on a aussi $|\cS|\leq 1$.)
\item
Soit $B$ un diviseur $K$-rationnel sur $X$ tel que
\beq
l(B)\geq 1.
\eeq
Supposons que pour tout $P\in\cS$ on ait $l(B-P)>l(B)-1$.
Alors
\beq
|\cS|\leq \deg B+1-l(B).
\eeq
(Si $\deg B=2g-1$ on a aussi $|\cS|\leq 1$.)
\end{enumerate}
\end{lemme}
\begin{proof}
C'est un r\'esultat classique~: pour a), voir par exemple
\cite{21sep}, lemme~10~; et b) lui est \'equivalent par
Riemann-Roch (poser $A=\Omega-B$ o\`u $\Omega$ est un diviseur canonique),
mais on peut aussi le prouver directement~: les hypoth\`eses impliquent
que $l(B')=l(B)$, o\`u $B'=B-\sum_{P\in\cS}P$, de sorte que
$l(B)=l(B')\leq 1+\deg B'=1+\deg B-|\cS|$, d'o\`u la conclusion.

(Si $\deg B=2g-1$, on a $g\geq1$ car $l(B)\geq 1$, et par ailleurs
$l(B-P)>l(B)-1$ si et seulement si
$B-P\sim\Omega$. Si on a aussi $l(B-P')>l(B)-1$,
alors $P'\sim P$, donc $P'=P$.)
\end{proof}

\medskip

Le c{\oe}ur technique de l'article est ce second lemme,
qui correspond \`a l'ordre~$2$, et qui raffine le lemme~12 de \cite{21sep}.

Pour tout $q>1$ et pour tout entier $n\geq2$ on pose
\beq
\begin{split}
G_q(n)&=\sum_{k=1}^{n-2}\frac{(q^{n-k}-1)(q^{n-k-1}-1)}{(q^{n}-1)(q^{n-1}-1)}\\
&=\frac{1}{q^2-1}-\frac{1-\frac{(q-1)n}{q^n-1}}{(q-1)(q^{n-1}-1)}
\end{split}
\eeq
(c'est une fonction croissante de $n$,
car chacun des termes dans la somme qui la d\'efinit l'est,
et qui tend vers $\frac{1}{q^2-1}$ \`a l'infini).

Alors~:

\begin{lemme}
\label{lemme2}
Soient $X$ une courbe de genre $g$ sur un corps $K$,
suppos\'e \emph{parfait},
et $\cS\subset X(K)$ un ensemble de points de $X$.
\begin{enumerate}[a)]
\item Soit $A$ un diviseur $K$-rationnel sur $X$ tel que $\deg A\geq -2$ et
\beq
i(A)=l(A)-(\deg A+1-g)\geq 2.
\eeq
Supposons que pour tout $P\in\cS$ on ait $l(A+2P)>l(A)$.
Alors
\beq
\label{miniPlucker2}
|\cS|\;\leq\; 3g+3\; +\; \deg A\;-\;3l(A).
\eeq
Si de plus $K$ est un corps fini $\F_q$,
on a aussi
\beq
|\cS|\leq\left(1+\frac{q^{i(A)-2}-1}{q^{i(A)}-1}\right)^{-1}(6g-6 \; -\;2\deg A\; -\;2\,G_q(i(A))\,|X(\F_q)|\,)
\eeq
et plus g\'en\'eralement, pour tout entier $w$ tel que
$2\leq w\leq i(A)$, 
\beq
\label{Plucker2}
\begin{split}
|\cS|\leq (i(A)-w)+\left(1+\frac{q^{w-2}-1}{q^{w}-1}\right)^{-1}(6g-6 \; -\;2&\deg A\; -\,4(i(A)-w)\;\\
&\;\;\;\;\;\;-\,2\,G_q(w)\,|X(\F_q)|\,).
\end{split}
\eeq 
\item
Soit $B$ un diviseur $K$-rationnel sur $X$ tel que $\deg B\leq 2g$ et
\beq
l(B)\geq 2.
\eeq
Supposons que pour tout $P\in\cS$ on ait $l(B-2P)>l(B)-2$.
Alors
\beq
\label{car_qcq}
|\cS|\;\leq\; 2\deg B\; +\; 2g+4\;-\;3l(B).
\eeq
Si de plus $K$ est un corps fini $\F_q$,
on a aussi
\beq
\label{grosseformulePlucker2facile}
|\cS|\leq\left(1+\frac{q^{l(B)-2}-1}{q^{l(B)}-1}\right)^{-1}(2\deg B \; +\; 2g-2 \; -\;2\,G_q(l(B))\,|X(\F_q)|\,)
\eeq
et plus g\'en\'eralement, pour tout entier $w$ tel que
$2\leq w\leq l(B)$, 
\beq
\label{grosseformulePlucker2}
\begin{split}
|\cS|\leq (l(B)-w)+\left(1+\frac{q^{w-2}-1}{q^{w}-1}\right)^{-1}(2\deg B \; +\; 2g-&2 \; -\,4(l(B)-w)\;\\
&\!\!-\,2\,G_q(w)\,|X(\F_q)|\,).
\end{split}
\eeq 
\end{enumerate}
\end{lemme}
\begin{proof}
Remarquons que a) et b) sont \'equivalents, par Riemann-Roch.
Il suffit donc de prouver b), et pour
ce faire, on proc\`ede en six \'etapes.

\smallskip

\textbf{\'Etape~1. Pr\'eliminaires et notations.}

\smallskip

Par hypoth\`ese, $\deg B\leq 2g$,
et $l(B)\geq 2$ donc $\deg B\geq1$,
ce qui implique, par le th\'eor\`eme de Clifford,
\beq
\label{Clifford}
l(B)\leq 1+\frac{\deg B}{2}.
\eeq

Pour tout $x\in\cL(B)\setminus\{0\}$
on note $E_x=B+\div x\geq 0$.
Si $V$ est un sous-$K$-espace vectoriel non nul de $\cL(B)$,
on d\'efinit son lieu base $E_V$ comme le $\pgcd$
des $E_x$ pour $x\in V\setminus\{0\}$~;
de fa\c{c}on \'equivalente, $E_V$ est le plus grand diviseur
tel que $V\subset\cL(B-E_V)\subset\cL(B)$.
Alors $V$ d\'efinit un morphisme $\phi_V:X\longto\PP^{\dim V-1}$
de degr\'e $\deg B-\deg E_V$.

\smallskip

\textbf{\'Etape~2. R\'eduction au cas s\'eparable.}
On montre le r\'esultat suivant
(voir aussi le d\'ebut de la preuve de \cite{21sep}, lemme~12)~:

\smallskip

\emph{Soit $V\subset\cL(B)$ de dimension $\dim V\geq2$.
Alors il existe $V'\subset\cL(B)$ de dimension $\dim V'=2$
tel que $\phi_{V'}:X\longto\PP^1$
soit s\'eparable et $E_{V'}\geq E_V$.}

\smallskip

En effet, remarquons tout d'abord que
quitte \`a remplacer $V$ par un sous-espace (ce qui ne peut que faire
augmenter $E_V$), on peut supposer $\dim V=2$.

Si maintenant $K$ est de caract\'eristique nulle,
ou plus g\'en\'eralement si $\phi_V$ est d\'ej\`a s\'eparable,
il n'y a rien \`a montrer
(on prend $V'=V$). Supposons donc $K$ de caract\'eristique
$p>0$, et $\phi_V$ de degr\'e d'ins\'eparabilit\'e $p^m>1$.
Choisissons une base $x,y\in V$ et posons $f=y/x\in K(X)$.
Dans l'\'equivalence de cat\'egories entre courbes alg\'ebriques
et corps de fonctions, $\phi_V$ est le morphisme
associ\'e \`a l'inclusion $K(f)\subset K(X)$, 
de degr\'e d'ins\'eparabilit\'e $p^m$. Puisque $K$ est suppos\'e
parfait, cela signifie qu'on peut \'ecrire $f=g^{p^m}$ 
avec $K(g)\subset K(X)$ s\'eparable. 

Par hypoth\`ese $\div x,\div y\geq -B$, donc
\beq
\div gx=\div x+\frac{1}{p^m}\div f=\left(1-\frac{1}{p^m}\right)\div x+\frac{1}{p^m}\div y\geq -B.
\eeq
Ainsi $gx\in\cL(B)$, et on note $V'\subset\cL(B)$
le sous-espace de base $x,gx$.
Alors $\phi_{V'}$ est le morphisme associ\'e \`a l'inclusion de
corps de fonctions $K(g)\subset K(X)$, donc est bien s\'eparable.
Il ne nous reste plus qu'\`a montrer $E_{V'}\geq E_V$.

Par d\'efinition, $E_V=\pgcd(E_x,E_y)$. Posons donc
$E_x=E_V+E_x'$ et $E_y=E_V+E_y'$, avec $E_x',E_y'\geq0$ \'etrangers.
Alors $E_{gx}=E_V+\left(1-\frac{1}{p^m}\right)E_x'+\frac{1}{p^m}E_y'$
donc
\beq
E_{V'}=\pgcd(E_x,E_{gx})=E_V+\left(1-\frac{1}{p^m}\right)E_x'\;\geq\; E_V
\eeq
ce qu'il fallait d\'emontrer.

\smallskip

\textbf{\'Etape~3. Une formule g\'en\'erale.} On montre~:

\smallskip

\emph{Soient $V\subset\cL(B)$ un sous-espace de dimension $\dim V\geq 2$,
et $E$ un diviseur sur $X$ tel que $0\leq E\leq E_V$. Alors
\beq
\label{formgen}
|\cS|\;\leq \; 2\deg B \; +\; 2g-2 \; -\;2\deg E \;+\;|\cS\cap\Supp E|.
\eeq}


Remarquons que la fonction
$E\;\mapsto \; 2\deg B \; +\; 2g-2 \; -\;2\deg E \;+\;|\cS\cap\Supp E|$
est une fonction d\'ecroissante de $E$. Il suffit donc de prouver
\eqref{formgen} lorsque $E=E_V$.
De plus, quitte \`a remplacer $V$ par $V'$ fourni par
l'\'etape 2,
on peut aussi supposer que $\dim V=2$ et que $\phi_V:X\longto\PP^1$
est s\'eparable.

Notons $R$ le diviseur de ramification de $\phi_V$.
Si $P\in\cS\setminus\Supp E_V$ on a $V\not\subset\cL(B-P)$,
donc $l(B-P)=l(B)-1$,
et par ailleurs $l(B-2P)>l(B)-2$, d'o\`u n\'ecessairement
\beq
\label{-2P=-P}
\cL(B-2P)=\cL(B-P).
\eeq
Si maintenant $t\in K(\PP^1)$
est une uniformisante en $\phi_V(P)$, alors $\phi_V^*t$ s'annule en $P$,
et par \eqref{-2P=-P}, son ordre d'annulation est au moins $2$.
Cela signifie que $\phi_V$ est ramifi\'e en $P$, d'o\`u
\beq
\deg R\geq |\cS\setminus\Supp E_V|.
\eeq
Par ailleurs la formule d'Hurwitz donne
\beq
2g-2=-2 \deg\phi_V + \deg R
\eeq
avec $\deg\phi_V=\deg B-\deg E_V$, d'o\`u il d\'ecoule \eqref{formgen}.

\smallskip

\textbf{\'Etape~4. Choix d'un sous-espace et preuve de \eqref{car_qcq}.}

\smallskip

Si $l(B)=2$, on prend $V=\cL(B)$ et $E=0$ dans \eqref{formgen},
ce qui prouve \eqref{car_qcq} dans ce cas particulier.

Si par ailleurs $|\cS|<l(B)-2$, alors \eqref{car_qcq} est cons\'equence
de \eqref{Clifford}.

On peut donc supposer $l(B)\geq 3$ et $|\cS|\geq l(B)-2$.
Notons $P_1,\dots,P_{|\cS|}$ les \'el\'ements de $\cS$,
et posons
\beq
E=2(P_1+\cdots+P_{l(B)-2}).
\eeq
Alors pour tout $i$ on a $l(B-2P_i)\geq l(B)-1$,
donc
\beq
l(B-E)\geq 2.
\eeq
Ainsi en posant $V=\cL(B-E)$ on satisfait aux hypoth\`eses de \eqref{formgen},
avec
\begin{itemize}
\item $\deg E=2(l(B)-2)$
\item $|\cS\cap\Supp E|=l(B)-2$ 
\end{itemize}
ce qui donne bien \eqref{car_qcq}.

\smallskip

\textbf{\'Etape~5. Preuve de \eqref{grosseformulePlucker2facile}.} 

\smallskip


On suppose maintenant $K=\F_q$.

Pour tout sous-espace $V\subset \cL(B)$ de dimension $\dim V=2$,
et pour tout point $P\in X(\F_q)$,
notons $\nu_{P,V}=v_P(E_V)\geq0$ la multiplicit\'e de $P$ dans $E_V$,
et posons
\beq
F_V=\sum_{P\in X(\F_q)}\nu_{P,V}P
\eeq
de sorte que $F_V$ est le plus grand diviseur de $E_V$ support\'e par $X(\F_q)$.
Posons alors
\beq
d_V=2\deg F_V-|\cS\cap\Supp F_V|=\sum_{P\in X(\F_q)}2\nu_{P,V}\;-\;|\{P\in\cS\;|\;\nu_{P,V}\geq1\}|.
\eeq
En remarquant que
\beq
\sum_{P\in X(\F_q)}2\nu_{P,V}=\sum_{P\in X(\F_q)}\left(2\sum_{k\geq1}\un_{\{\nu_{P,V}\geq k\}}\right)
\eeq
et
\beq
\begin{split}
|\{P\in\cS\;|\;\nu_{P,V}\geq1\}|=\sum_{P\in\cS}\un_{\{\nu_{P,V}\geq 1\}}
\end{split}
\eeq
on peut aussi \'ecrire
\beq
\begin{split}
d_V=\sum_{P\in X(\F_q)\setminus\cS}&\left(2\sum_{k\geq1}\un_{\{\nu_{P,V}\geq k\}}\right)\\
&+\sum_{P\in\cS}\left(\un_{\{\nu_{P,V}\geq 1\}}+2\sum_{k\geq2}\un_{\{\nu_{P,V}\geq k\}}\right)
\end{split}
\eeq

On consid\`ere maintenant $d_V$ et les $\nu_{P,V}$ comme des variables
al\'eatoires, en supposant
que $V$ est tir\'e selon
la loi de probabilit\'e uniforme
sur l'ensemble des sous-espaces de dimension $2$ de $\cL(B)$.
La formule pr\'ec\'edente permet alors de calculer l'esp\'erance de $d_V$~:
\beq
\label{EdV}
\begin{split}
\mathbf{E}[d_V]=\sum_{P\in X(\F_q)\setminus\cS}&\left(2\sum_{k\geq1}\mathbf{P}[\nu_{P,V}\geq k]\right)\\
&+\sum_{P\in\cS}\left(\mathbf{P}[\nu_{P,V}\geq 1]+2\sum_{k\geq2}\mathbf{P}[\nu_{P,V}\geq k]\right).
\end{split}
\eeq
Remarquons que, par d\'efinition de $E_V$,
on a l'\'equivalence~:
\beq
\nu_{P,V}\geq k \qquad\equivaut\qquad V\subset\cL(B-kP).
\eeq
Puisque le nombre de sous-espaces de dimension $2$ dans un espace de
dimension $d$ est $\frac{(q^d-1)(q^{d-1}-1)}{(q^2-1)(q-1)}$,
ceci donne
\beq
\begin{split}
\mathbf{P}[\nu_{P,V}\geq k]&=\frac{(q^{l(B-kP)}-1)(q^{l(B-kP)-1}-1)}{(q^{l(B)}-1)(q^{l(B)-1}-1)}\\
&=g_q(\,l(B),\;l(B)-l(B-kP)\,)\phantom{\sum}
\end{split}
\eeq
o\`u la fonction
\beq
g_q(x,y)=\frac{(q^{x-y}-1)(q^{x-y-1}-1)}{(q^{x}-1)(q^{x-1}-1)}
\eeq
est croissante en $x$ 
(pour $x\geq2$) et d\'ecroissante en $y$ (pour $0\leq y<x$).
Pour tout $P\in X(\F_q)$ et tout $k\geq 1$ on a $l(B)-l(B-kP)\leq k$,
de sorte que
\beq
\mathbf{P}[\nu_{P,V}\geq k]\geq g_q(l(B),k),
\eeq
et si en outre $P\in\cS$
et $k\geq2$, on a alors m\^eme $l(B)-l(B-kP)\leq k-1$, d'o\`u
dans ce cas~:
\beq
\mathbf{P}[\nu_{P,V}\geq k]\geq g_q(l(B),k-1).
\eeq
Reportant tout ceci dans \eqref{EdV} on trouve
\beq
\begin{split}
\mathbf{E}[d_V]&\geq(|X(\F_q)|-|\cS|)\sum_{k=1}^{l(B)-2}2\,g_q(l(B),k)\\
&\phantom{\geq(|X(\F_q)|-|\cS}+|\cS|\left(g_q(l(B),1)+\sum_{k=2}^{l(B)-1}2\,g_q(l(B),k-1)\right)\\
&\phantom{\geq(}=|X(\F_q)|\sum_{k=1}^{l(B)-2}2\,g_q(l(B),k)\;+\;|\cS|\,g_q(l(B),1)
\end{split}
\eeq
soit
\beq
\mathbf{E}[d_V]\geq2\,G_q(l(B))\,|X(\F_q)|\;+\;\frac{q^{l(B)-2}-1}{q^{l(B)}-1}\,|\cS|.
\eeq

Puisque cette in\'egalit\'e est vraie en moyenne, il existe
au moins un $V$ pour lequel elle est v\'erifi\'ee.
Choisissons donc un tel $V$, et appliquons \eqref{formgen}
avec $E=F_V$. 
On trouve alors
\beq
\begin{split}
|\cS|\;&\leq \; 2\deg B \; +\; 2g-2 \; -\;d_V\\
&\leq \; 2\deg B \; +\; 2g-2 \; -\;2\,G_q(l(B))\,|X(\F_q)|\;-\;\frac{q^{l(B)-2}-1}{q^{l(B)}-1}\,|\cS|
\end{split}
\eeq
ce qui donne bien
\eqref{grosseformulePlucker2facile}.

\smallskip

\textbf{\'Etape~6. Preuve de \eqref{grosseformulePlucker2}.} 

\smallskip

Soit maintenant $w$ un entier tel que $2\leq w\leq l(B)$.

Commen\c{c}ons par montrer que le second terme de la somme
dans \eqref{grosseformulePlucker2} est positif~:
\beq
2\deg B \, + \, 2g-2 \, -4(l(B)-w)\, - \,2\,G_q(w)\,|X(\F_q)|\;\geq\; 0.
\eeq
En effet, 
par \eqref{Clifford} on a $2\deg B-4l(B)\geq -4$,
par hypoth\`ese on a $w\geq2$ donc $4w\geq8$,
tandis que $G_q(w)\leq\frac{1}{q^2-1}$,
et $|X(\F_q)|\leq q+1+2g\sqrt{q}$ par la borne de Weil.
La quantit\'e qui nous int\'eresse est donc minor\'ee par
$2g+2-2\frac{q+1+2g\sqrt{q}}{q^2-1}=2g(1-\frac{2\sqrt{q}}{q^2-1})+2(1-\frac{1}{q-1})$,
ce qui est bien positif, puisque $q\geq2$.

On en d\'eduit que, si $w$ est choisi de fa\c{c}on que
$|\cS|\leq l(B)-w$, alors \eqref{grosseformulePlucker2} est bien v\'erifi\'ee.

Supposons donc maintenant $|\cS|>l(B)-w$.

Notant toujours $P_1,\dots,P_{|\cS|}$ les \'el\'ements de $\cS$,
on pose
\beq
B'=B-2(P_1+\cdots+P_{l(B)-w})
\eeq
et
\beq
\cS'=\{P_{l(B)-w+1},\dots,P_{|\cS|}\}.
\eeq
On a alors $l(B')\geq w\geq2$, et $l(B'-2P)>l(B')-2$ pour tout $P\in\cS'$.
Appliquant \eqref{grosseformulePlucker2facile} \`a $B'$ et $\cS'$, on trouve
alors
\beq
\begin{split}
|\cS'|&=|\cS|-(l(B)-w)\\
&\leq\left(1+\frac{q^{l(B')-2}-1}{q^{l(B')}-1}\right)^{-1}(2\deg B' \; +\; 2g-2 \; -\;2\,G_q(l(B'))\,|X(\F_q)|\,)\\
&\leq\left(1+\frac{q^{w-2}-1}{q^{w}-1}\right)^{-1}(2\deg B-4l(B)+4w\,+\,2g-2\,-\,2G_q(w)|X(\F_q)|)
\end{split}
\eeq
ce qu'il fallait d\'emontrer.
\end{proof}


\begin{remarque}
\label{remcarpos}
En caract\'eristique nulle on pourrait g\'en\'eraliser ces r\'esultats
\`a un ordre quelconque $s$, en montrant que pour tout diviseur $A$ de degr\'e
$\deg A\geq -s$ tel que $i(A)\geq s$, le nombre de points $P$ tels que
$l(A+sP)>l(A)$ est major\'e essentiellement par $s^2g$.
Ces points sont en effet ceux o\`u un certain wronskien s'annule,
et la majoration est donn\'ee par le degr\'e du faisceau inversible
dont ce wronskien est une section.

Comme expliqu\'e dans la remarque finale de \cite{21sep},
cela n'est pas vrai cependant en caract\'eristique positive
d\`es lors que $s\geq3$. En effet, pour que le raisonnement indiqu\'e
ci-dessus soit valable, il faut s'assurer que ce wronskien n'est pas
identiquement nul. Pour $s=2$, ceci est rendu possible par
l'\'etape~2 de la preuve, qui ne se g\'en\'eralise pas \`a l'ordre
sup\'erieur.

Ceci peut s'interpr\'eter, en termes de th\'eorie des sauts de
Weierstrass, par l'existence, en caract\'eristique positive,
de diviseurs dont la suite des invariants d'Hermite
g\'en\'erique (\cite{Laksov}\cite{SV})
commence bien par $\epsilon_0=0$ et $\epsilon_1=1$
(ce qui \'equivaut au r\'esultat de s\'eparabilit\'e prouv\'e
\`a l'\'etape~2),
mais v\'erifie $\epsilon_{s-1}>s-1$ pour $s\geq3$.
De tels diviseurs sont ``rares'' (\cite{Neeman}), mais pour les
applications qu'on a en vue par la suite, il n'est pas clair qu'on
puisse les \'eviter.
\end{remarque}

\section{Construction de diviseurs ordinaires}
\label{section2}

\begin{definition}
\label{defordinaire}
Soit $X$ une courbe de genre $g\geq1$ sur un corps $K$.
On dit qu'un diviseur $D$ sur $X$ est \emph{ordinaire}
s'il v\'erifie
\beq
l(D)=\max(0,\deg D+1-g),
\eeq
autrement dit, $l(D)=0$ si $\deg D\leq g-1$,
et $l(D)=\deg D+1-g$ si $\deg D\geq g-1$.

Dans le cas contraire on dit que $D$ est \emph{exceptionnel}.
\end{definition}

On remarque que ces propri\'et\'es ne d\'ependent que de la classe
d'\'equivalence lin\'eaire de $D$. De plus, $D$ est ordinaire
(resp. exceptionnel) si et seulement si $\Omega-D$ l'est, o\`u
$\Omega$ est un diviseur canonique.
En fait, $D$ est exceptionnel si et seulement si $D$ et $\Omega-D$
sont tous les deux sp\'eciaux (ce qui implique $0\leq \deg D\leq 2g-2$).
Si $0\leq d\leq g-1$, les classes de diviseurs exceptionnels de degr\'e
$d$ sont param\'etr\'ees par la sous-vari\'et\'e $W_d\subset J$
introduite au d\'ebut de ce texte. 

On s'int\'eresse ici au probl\`eme suivant~:

\begin{probleme}
\label{pb3}
Soient $X$ une courbe alg\'ebrique sur un corps $K$,
et $r\geq 1$ un entier.
\'Etant donn\'es des entiers $s_1,\dots,s_r\geq1$,
des diviseurs $K$-rationnels $T_1,\dots,T_r$ sur $X$,
et un entier $d\in\Z$, construire
un diviseur $K$-rationnel $D$ sur $X$ de degr\'e $\deg D=d$,
tel que les diviseurs $s_1D-T_1,\dots,s_rD-T_r$ soient ordinaires.
\end{probleme}

Indiquons comment le probl\`eme~\ref{pb1} peut se ramener \`a ce
probl\`eme~\ref{pb3}.
Dans le probl\`eme~\ref{pb1}, on se donne des entiers $k_1,\dots,k_r$,
positifs ou n\'egatifs, et on veut trouver un diviseur
$D$ de degr\'e $d$ tel que
\beq
l(k_1D-G_1)=\cdots=l(k_rD-G_r)=0.
\eeq
Une condition n\'ecessaire sur $d$ pour que ce soit possible
est donn\'ee par \eqref{max<=g-1}, et peut se traduire en
\beq
\label{d-<d<d+}
\max_{k_i<0}\left\lceil\frac{g-1+n_i}{k_i}\right\rceil=d^-\;\leq\; d\;\leq\; d^+=\min_{k_i>0}\left\lfloor\frac{g-1+n_i}{k_i}\right\rfloor
\eeq
o\`u $n_i=\deg G_i$.

Ainsi, une condition n\'ecessaire pour que le probl\`eme~\ref{pb1}
soit r\'esoluble est que
\beq
d^-\leq d^+.
\eeq
Supposons cette condition v\'erifi\'ee, et choisissons $d$ 
v\'erifiant \eqref{d-<d<d+}.
Posons~:
\begin{itemize}
\item $s_i=k_i$ et $T_i=G_i$ si $k_i>0$
\item $s_i=-k_i$ et $T_i=-\Omega-G_i$ si $k_i<0$
\end{itemize}
o\`u, comme pr\'ec\'edemment, $\Omega$ est un diviseur canonique sur $X$.

Supposons alors qu'on trouve $D$ de degr\'e $\deg D=d$ solution
du probl\`eme~\ref{pb3} pour ces choix de $s_i$ et $T_i$. Ainsi,
tous les $s_iD-T_i$ sont suppos\'es ordinaires. Alors~:
\begin{itemize}
\item si $k_i>0$, on a $\deg(s_iD-T_i)\leq g-1$ par \eqref{d-<d<d+},
donc $l(s_iD-T_i)=0$, donc $l(k_iD-G_i)=0$
\item si $k_i<0$, on a $\deg(s_iD-T_i)\geq g-1$ par \eqref{d-<d<d+},
donc $l(s_iD-T_i)=\deg(s_iD-T_i)+1-g$, d'o\`u par Riemann-Roch
$l(k_iD-G_i)=0$.
\end{itemize}
Ainsi $D$ est bien solution du probl\`eme~\ref{pb1}.

\medskip

On explique maintenant comment r\'esoudre le probl\`eme~\ref{pb3}
lorsque $K$ est parfait,
que les $s_i$ valent $1$ ou $2$, et que $X$ a suffisamment
de points (on supposera toujours $g\geq1$).

On introduit deux fonctions $f_{1,X}$ et $f_{2,X}$ d\'efinies
sur $\Z$, comme suit.
Tout d'abord,
\beq
f_{1,X}(a)=
\begin{cases}
1 & \textrm{si $a=-1$}\\
g & \textrm{si $0\leq a\leq g-2$}\\
0 & \textrm{sinon.}
\end{cases}
\eeq
La d\'efinition de $f_{2,X}$ est un peu plus compliqu\'ee.
On pose
\beq
f_{2,X}(g-2)=g
\eeq
puis, si $-2\leq a\leq g-3$ et que $K=\F_q$ est un corps fini~:
\beq
\begin{split}
f_{2,X}(a)=\!\!\!\min_{2\leq w\leq g-1-a}\!\lfloor(g\!-\!1\!-\!a\!-\!w)+\!\left(1+\frac{q^{w-2}\!-\!1}{q^{w}-1}\right)^{-1}\!\!\!(2g-&2+2a+4w\\
&\;-\!2G_q(w)|X(\F_q)|)\rfloor
\end{split}
\eeq
ou bien, si $-2\leq a\leq g-3$ et que $K$ est infini~:
\beq
f_{2,X}(a)=3g+3+a,
\eeq
et enfin $f_{2,X}(a)=0$ si $a<-2$ ou $a>g-2$.

\begin{lemme}
\label{recapitulatif}
Avec les notations qui pr\'ec\`edent, soient $A$ un diviseur sur $X$
et $\cS\subset X(K)$ un ensemble de points.
Soit aussi $s\in\{1,2\}$. On suppose que $A$ est ordinaire, mais
que $A+sP$ est exceptionnel pour tout $P\in\cS$.
Alors
\beq
|\cS|\leq f_{s,X}(\deg A).
\eeq
\end{lemme}
\begin{proof}
On pose $a=\deg A$, et on distingue selon que $s=1$ ou $s=2$.
\begin{enumerate}[(i)]
\item Supposons $s=1$.
\begin{itemize}
\item Si $a\leq -2$ ou $a\geq 2g-2$, alors $A+P$ est toujours
ordinaire, donc $\cS$ est vide, et on a bien $f_{1,X}(a)=0$.
\item Si $g-1\leq a\leq 2g-3$, alors $A$ ordinaire signifie
$l(A)=a+1-g$, donc $A+P$ est ordinaire par Riemann-Roch,
et on conclut de m\^eme.
\item Si $-1\leq a\leq g-2$, alors $A$ ordinaire signifie $l(A)=0$,
et $A+P$ exceptionnel signifie $l(A+P)=1$. 
On conclut gr\^ace au lemme~\ref{lemme1}.a.
\end{itemize}
\item Supposons $s=2$.
\begin{itemize}
\item Si $a\leq -3$ ou $a\geq 2g-1$, alors $A+2P$ est toujours
ordinaire, donc $\cS$ est vide, et on a bien $f_{2,X}(a)=0$.
\item Si $g-1\leq a\leq 2g-2$, alors $A$ ordinaire signifie
$l(A)=a+1-g$, donc $A+2P$ est ordinaire par Riemann-Roch,
et on conclut de m\^eme.
\item Si $a=g-2$, alors $A$ ordinaire signifie $l(A)=0$,
et $A+2P$ exceptionnel signifie $l(A+2P)=2$,
donc $l(A+P)=1$. Le lemme~\ref{lemme1}.a donne alors bien
$|\cS|\leq g=f_{2,X}(a)$.
\item Enfin si $-2\leq a\leq g-3$, alors $A$ ordinaire signifie $l(A)=0$,
et $A+2P$ exceptionnel signifie $l(A+2P)\geq 1$.
On conclut alors gr\^ace au lemme~\ref{lemme2}.a,
(en remarquant notamment, dans \eqref{Plucker2}, qu'ici $i(A)=g-1-a$).
\end{itemize}
\end{enumerate}
\end{proof}

\begin{proposition}
\label{prop-constr}
Soient $X$ une courbe de genre $g$ sur
un corps $K$ parfait, et $r\geq 1$ un entier.
On suppose donn\'es des entiers $s_1,\dots,s_r\in\{1,2\}$,
des diviseurs $K$-rationnels $T_1,\dots,T_r$ sur $X$,
et un entier $d\in\Z$.
On notera $t_i=\deg T_i$.
Soit aussi $D_0$ un diviseur $K$-rationnel sur $X$,
de degr\'e $\deg D_0=d_0\leq d$, tel que les
$s_iD_0-T_i$ soient ordinaires
(c'est vrai par exemple si $d_0\leq\min_i\left\lfloor\frac{t_i-1}{s_i}\right\rfloor$).
Soit enfin $\cS\subset X(K)$
un ensemble de points tel que
\beq
\label{borneS}
|\cS|>\max_{d_0\leq d'<d}\sum_{i=1}^rf_{s_i,X}(s_id'-t_i).
\eeq
Alors il existe un diviseur $K$-rationnel $D$ sur $X$,
de degr\'e $\deg D=d$, tel que les
$s_iD-T_i$ soient ordinaires.
De plus on peut choisir $D$ de fa\c{c}on que $D-D_0$ soit
effectif et \`a support dans $\cS$.
\end{proposition}
\begin{proof}
On construit $D$ de proche en proche.
Soit $d'$ tel que $d_0\leq d'<d$, et supposons construit $D'$ de degr\'e
de degr\'e $\deg D'=d'$, tel que les
$s_iD'-T_i$ soient ordinaires, et tel que $D'-D_0$ soit
effectif et \`a support dans $\cS$.
Le lemme~\ref{recapitulatif} appliqu\'e avec $s=s_i$
et $A=s_iD'-T_i$ montre qu'il y a au plus 
$f_{s_i,X}(s_id'-t_i)$ points tels que $s_i(D'+P)-T_i$
soit exceptionnel. Gr\^ace \`a \eqref{borneS},
on peut donc trouver $P\in\cS$ tel que les $s_i(D'+P)-T_i$ soient ordinaires,
et par construction $(D'+P)-D_0$ est encore
effectif et \`a support dans $\cS$.
On conclut alors par r\'ecurrence sur $d'$.
\end{proof}

\begin{remarque}
La restriction impos\'ee aux $s_i$ de valoir $1$ ou $2$ r\'esulte
des particularit\'es de la caract\'eristique positive indiqu\'ees
dans la remarque~\ref{remcarpos}.
En caract\'eristique nulle, on pourrait donc
\'enoncer une variante de la proposition~\ref{prop-constr}
valable sans restriction sur les $s_i$.
Pour les applications auxquelles on s'int\'eresse ici, cela n'est
cependant pas n\'ecessaire.
\end{remarque}

Pour la suite il pourra \^etre utile de pr\'eciser quelques propri\'et\'es
des $f_{s,X}$.

Le corps $K$ \'etant donn\'e, 
pour tout r\'eel
$\nu$ on introduit deux fonctions $\phi_{1,\nu}$ et $\phi_{2,\nu}$
d\'efinies sur $\R$, comme suit~:
\beq
\phi_{1,\nu}(\alpha)=
\begin{cases}
1 & \textrm{si $0\leq\alpha\leq1$}\\
0 & \textrm{si $\alpha<0$ ou $\alpha>1$}
\end{cases}
\eeq
et
\beq
\label{defphi2nu}
\phi_{2,\nu}(\alpha)=
\begin{cases}
\frac{3q^2+1}{q^2+1}+\frac{q^2-1}{q^2+1}\alpha-\frac{2q^2}{q^4-1}\nu & \textrm{si $0\leq\alpha<1$ et $|K|=q<+\infty$}\\
3+\alpha & \textrm{si $0\leq\alpha<1$ et $K$ infini}\\
4 & \textrm{si $\alpha=1$}\\
0 & \textrm{si $\alpha<0$ ou $\alpha>1$}
\end{cases}
\eeq
(remarquons que l'expression pour $K$ infini peut s'obtenir
comme limite du cas fini
quand $q\longto\infty$).

\begin{lemme}
\label{ptesdiverses}
Avec les notations pr\'ec\'edentes, soit $s\in\{1,2\}$.
\begin{enumerate}[a)]
\item
Pour toute courbe $X$ sur $K$, la fonction $f_{s,X}(a)$ est une fonction
croissante de $a$ pour $a\leq g-1-s$.

Son maximum sur $\Z$ est $f_{s,X}(g-1-s)=s^2g$.
\item
Soient $\alpha$ et $\nu$ deux r\'eels, avec $\alpha\not\in\{0,1\}$.
Pour tout entier $j\geq1$, on se donne une courbe $X^{(j)}$ de genre
$g^{(j)}$ sur $K$, et un entier $a^{(j)}\in\Z$. On suppose que,
lorsque $j$ tend vers l'infini, on a $g^{(j)}\longto\infty$, et~:
\begin{itemize}
\item $\frac{a^{(j)}}{g^{(j)}}\longto\alpha$
\item $\liminf\frac{|X^{(j)}(K)|}{g^{(j)}}\geq\nu$.
\end{itemize}
Alors, quand $j$ tend vers l'infini,
\beq
\limsup\frac{f_{s,X^{(j)}}(a^{(j)})}{g^{(j)}}\leq\phi_{s,\nu}(\alpha).
\eeq
\end{enumerate}
\end{lemme}
\begin{proof}
Pour a), le seul point non trivial est la croissance de $f_{2,X}$
sur l'intervalle $-2\leq a\leq g-3$ lorsque $|K|=q<\infty$.
Sous cette hypoth\`ese, posons
\beq
\label{ftilde}
\widetilde{f_{2,X}}(a,w)=(g\!-\!1\!-\!a\!-\!w)+\!\left(1+\frac{q^{w-2}\!-\!1}{q^{w}-1}\right)^{-1}\!\!\!(2g\!-\!2+2a+4w-\!2G_q(w)|X(\F_q)|)
\eeq
de sorte que
$f_{2,X}(a)=\min_{2\leq w\leq g-1-a}\lfloor\widetilde{f_{2,X}}(a,w)\rfloor$.

Pour $w\geq 2$, on a
$\frac{q^{w-2}\!-1}{q^{w}-1}\leq\frac{1}{q^2}$,
d'o\`u
$\left(1+\frac{q^{w-2}\!-1}{q^{w}-1}\right)^{-1}\geq1-\frac{1}{q^2+1}$,
de sorte qu'\`a $w$ fix\'e, $\widetilde{f_{2,X}}(a,w)$ est fonction croissante
de $a$.
On conclut alors en passant au $\min$ sur $w$.

De m\^eme pour b), le seul cas non trivial est celui
o\`u $s=2$ et $0\leq\alpha<1$, avec $|K|=q<\infty$.
On se placera donc sous cette hypoth\`ese.

Puisque $\frac{a^{(j)}}{g^{(j)}}\longto\alpha<1$ on peut, pour tout $j$
assez grand, choisir un entier $w^{(j)}$
v\'erifiant $2\leq w^{(j)}\leq g^{(j)}-1-a^{(j)}$, et de fa\c{c}on que~:
\begin{itemize}
\item $w^{(j)}\longto\infty$
\item $\frac{w^{(j)}}{g^{(j)}}\longto 0.$
\end{itemize}
Alors $\frac{q^{w^{(j)}-2}\!-1}{q^{w^{(j)}}-1}\longto\frac{1}{q^2}$
et $G_q(w^{(j)})\longto\frac{1}{q^2-1}$, et en rempla\c{c}ant
dans \eqref{ftilde},
\beq
\limsup\frac{\widetilde{f_{2,X^{(j)}}}(a^{(j)},w^{(j)})}{g^{(j)}}\leq 1-\alpha+\left(1-\frac{1}{q^2+1}\right)\left(2+2\alpha-2\frac{\nu}{q^2-1}\right).
\eeq
On conclut alors puisque
$\limsup\frac{f_{2,X^{(j)}}(a^{(j)})}{g^{(j)}}\leq\limsup\frac{\widetilde{f_{2,X^{(j)}}}(a^{(j)},w^{(j)})}{g^{(j)}}$.
\end{proof}

Gr\^ace \`a ce lemme
on pourrait, si on le souhaitait, \'enoncer une version ``asymptotique''
de la proposition pr\'ec\'edente. Cependant, le faire
de fa\c{c}on optimale
n\'ecessiterait, dans le cas g\'en\'eral,
une distinction de cas assez fastidieuse
li\'ee aux discontinuit\'es des $\phi_{s,\nu}$ en $0$ et $1$.
On se contentera donc de le faire dans le cadre des deux applications
qui nous int\'eressent,
\`a savoir la construction de codes lin\'eaires intersectants
($r=1$ et $s_1=2$),
et celle d'algorithmes de multiplication dans les corps finis de
faible complexit\'e bilin\'eaire 
($r=2$ et $s_1=1$, $s_2=2$).

Comme indiqu\'e dans l'introduction, le cas des syst\`emes de partage
de secret avec propri\'et\'e de multiplication semble en revanche moins
bien s'y pr\^eter ($r$ variable).





\section{Application aux codes lin\'eaires intersectants}

On commence par rappeler quelques d\'efinitions et r\'esultats
de \cite{ITW2010}\cite{21sep}\cite{Xing2002}.

Un code lin\'eaire intersectant de param\`etres
$[n,k]$ sur un corps $K$ est un sous-espace vectoriel $C\subset K^n$ de
dimension $k$ tel que, pour tous $c,c'\in C$ non nuls,
les supports de $c$ et $c'$ s'intersectent
(\ie il existe $i$ tel que $c_ic_i'\neq0$).

On note $R_K$ le rendement asymptotique maximal,
\cad la $\limsup$ du rapport $k/n$,
qu'un tel code peut atteindre.
Si $K=\F_q$, on notera aussi $R_q$ pour $R_{\F_q}$.

Soit $X$ une courbe de genre $g$ sur un corps $K$.
Soient $P_1,\dots,P_n\in X(K)$ des points de $X$,
deux \`a deux distincts,
et pour chaque $i$, soit $z_i$ une uniformisante en $P_i$.
On notera $G$ la donn\'ee de ces $P_i$ et $z_i$ (par
abus de notation, on \'ecrira aussi parfois
$G=P_1+\cdots+P_n$,
le diviseur sur $X$ somme des $P_i$). Alors~:

\begin{definition}
\label{defevgen}
Pour tout diviseur $K$-rationnel $D$ sur $X$, le code
de Goppa g\'en\'eralis\'e $C(G,D)\subset K^n$
est l'image de l'application d'\'evaluation
\beq
\begin{array}{cccc}
\ev_{G,D}: & \cL(D) & \longto & K^n \\
& f & \mapsto & ((z_1^{\nu_1}f)(P_1),\dots,(z_n^{\nu_n}f)(P_n))
\end{array}
\eeq
o\`u $\nu_i=v_{P_i}(D)$ est la valuation de $D$ en $P_i$.
\end{definition}
Lorsque $D$ varie, la collection des $\ev_{G,D}$
d\'efinit un morphisme de $K$-alg\`ebres
\beq
\ev_G:\bigoplus_{D\in\Div_{\F_q}(X)}\cL(D)\longto K^n,
\eeq
o\`u $\bigoplus_D\cL(D)$ est munie de sa structure de $K$-alg\`ebre
gradu\'ee provenant de la multiplication dans le corps de fonctions
$K(X)$, et o\`u $K^n$ est muni de la
multiplication composante par composante.

De fa\c{c}on plus concr\`ete, pour tous diviseurs $D$ et $D'$,
le diagramme
\beq
\label{morphisme-algebre}
\begin{CD}
\cL(D)\times\cL(D') @>\ev_{G,D}\times\ev_{G,D'}>> K^n\times K^n\\
@VVV @VVV\\
\cL(D+D') @>\ev_{G,D+D'}>> K^n
\end{CD}
\eeq
commute,
envoyant
 $(f,f')\in\cL(D)\times\cL(D')$
sur
\beq
((z_1^{\nu_1+\nu_1'}ff')(P_1),\dots,(z_n^{\nu_n+\nu_n'}ff')(P_n))\in K^n
\eeq
(o\`u $\nu_i=v_{P_i}(D)$ et $\nu_i'=v_{P_i}(D')$).

Alors~:

\begin{proposition}[crit\`ere de Xing, \cite{Xing2002} Th.~3.5, ou \cite{21sep} Th.~7]
Avec les notations qui pr\'ec\`edent, supposons $\deg D<n=\deg G$
et
\beq
l(2D-G)=0.
\eeq
Alors $C(G,D)\subset K^n$ est 
un code lin\'eaire intersectant, de dimension $l(D)$.
\end{proposition}

La preuve r\'esulte de la commutativit\'e
du diagramme \eqref{morphisme-algebre}, avec $D'=D$,
et du fait que $\ker\ev_{G,2D}=\cL(2D-G)$.
Pour plus de d\'etails, voir par exemple \cite{21sep}.

On en d\'eduit~:

\begin{proposition}
Soient $X$ une courbe de genre $g$ sur un corps $K$, suppos\'e parfait,
$\cS\subset X(K)$ un ensemble de points de $X$,
et $n$ un entier naturel tel que $g\leq n\leq|X(K)|$.
Soit $d$ un entier naturel v\'erifiant
\beq
\label{d<n+g-1/2}
d\leq\frac{n+g-1}{2}
\eeq
et
\beq
\label{X>f2(2d-2-n)}
|\cS|>f_{2,X}(2d-2-n).
\eeq
Alors il existe 
des diviseurs 
$D$ de degr\'e $d$, \`a support dans $\cS$,
et $G$ de degr\'e $n$ sur $X$, tels que le code
$C=C(G,D)\subset K^n$ soit intersectant et
de dimension $\dim C\geq d+1-g$.

En particulier, si $|X(K)|>4g$, il existe 
un code lin\'eaire intersectant $C\subset K^n$ de dimension
$\dim C\geq\left\lfloor\frac{n+g-1}{2}\right\rfloor+1-g\geq\frac{n-g}{2}$.
\end{proposition}
\begin{proof}
Choisissons $P_1,\dots,P_n\in X(K)$ deux \`a deux distincts,
avec $P_1\in\cS$,
et appliquons la proposition \ref{prop-constr}
avec $r=1$, $s_1=2$, $T_1=G=P_1+\cdots+P_n$,
$d_0=\left\lfloor\frac{n-1}{2}\right\rfloor$, $D_0=d_0P_1$.
Le lemme~\ref{ptesdiverses}.a et \eqref{X>f2(2d-2-n)}
impliquent que \eqref{borneS} est v\'erifi\'ee, et
on en d\'eduit l'existence d'un diviseur $D$ de degr\'e $d$
\`a support dans $\cS$ tel
que $2D-G$ soit ordinaire. Alors par \eqref{d<n+g-1/2} on a
$2d-n\leq g-1$ donc $l(2D-G)=0$, et aussi $d<n$, et le crit\`ere de Xing montre
que le code $C(G,D)$ v\'erifie les conditions demand\'ees.

La derni\`ere assertion r\'esulte du fait que $f_{2,X}(g-3)=4g$,
de sorte qu'on peut prendre $\cS=X(K)$ et
$d=\lfloor\frac{n+g-1}{2}\rfloor$.
\end{proof}

\begin{corollaire}
Soient $K$ un corps parfait, et $\nu>1$ un r\'eel.
Pour tout entier $j\geq1$, on se donne une courbe $X^{(j)}$ de genre
$g^{(j)}$ sur $K$, et on suppose que, lorsque $j$ tend vers l'infini,
on a $g^{(j)}\longto\infty$ et
\beq
\liminf\frac{|X^{(j)}(K)|}{g^{(j)}}\geq\nu.
\eeq
Alors pour $j$ assez grand, il existe des diviseurs
$D^{(j)}$ et $G^{(j)}$ sur $X^{(j)}$ tels que le code $C(G^{(j)},D^{(j)})$
soit intersectant de param\`etres $[n^{(j)},k^{(j)}]$ avec, lorsque
$j$ tend vers l'infini, $\frac{n^{(j)}}{g^{(j)}}\longto\nu$
et
\beq
\liminf\frac{k^{(j)}}{n^{(j)}}\geq\min\left(1-\frac{5}{2\nu},\frac{1}{2}-\frac{1}{2\nu}\right).
\eeq
Lorsque $|K|=q<+\infty$, cette derni\`ere condition peut \^etre am\'elior\'ee en
\beq
\label{Rqnu}
\liminf\frac{k^{(j)}}{n^{(j)}}\geq\min\left(\,1-\frac{5}{2\nu}+\frac{2-\frac{2}{\nu}+\frac{1}{q^2-1}}{q^2-1}\,,\;\frac{1}{2}-\frac{1}{2\nu}\,\right).
\eeq
\end{corollaire}
\begin{proof}
Pour tout $\epsilon>0$,
posons $\nu_\epsilon=\nu-\epsilon$.
Choisissons un tel $\epsilon$ suffisamment petit, de sorte que
$\nu>\nu_\epsilon>\nu_{2\epsilon}>1$.

Pour tout $j$ suffisamment grand, choisissons un entier $n_\epsilon^{(j)}$,
en faisant en sorte que
$\frac{n_\epsilon^{(j)}}{g^{(j)}}\longto\nu_\epsilon$ lorsque $j$ tend vers
l'infini.
Pour $j$ assez grand on aura donc $g^{(j)}\leq n_\epsilon^{(j)}\leq|X^{(j)}(K)|$.

Soit $\delta_\epsilon$ le plus grand r\'eel v\'erifiant les in\'egalit\'es
\beq
\label{d<n+g-1/2asympt}
\delta_\epsilon\leq\frac{\nu_{2\epsilon}+1}{2}
\eeq
et
\beq
\label{X>f2(2d-2-n)asympt}
\nu_{\epsilon}\geq\phi_{2,\nu_{\epsilon}}(2\delta_\epsilon-\nu_{\epsilon}).
\eeq
Autrement dit, si $K$ est infini,
\beq
\delta_\epsilon=\min\left(\nu_{\epsilon}-\frac{3}{2},\frac{\nu_{2\epsilon}+1}{2}\right)
\eeq
ou bien, si $|K|=q<+\infty$,
\beq
\delta_\epsilon=\min\left(\nu_{\epsilon}-\frac{3}{2}+\frac{2\nu_{\epsilon}-2+\frac{\nu_{\epsilon}}{q^2-1}}{q^2-1}\,,\;\frac{\nu_{2\epsilon}+1}{2}\right).
\eeq
Pour tout $j$ suffisamment grand, choisissons un entier $d_\epsilon^{(j)}$,
en faisant en sorte que
$\frac{d_\epsilon^{(j)}}{g^{(j)}}\longto\delta_\epsilon$ lorsque $j$ tend vers
l'infini. Alors, 
pour $j$ assez grand, on aura par \eqref{d<n+g-1/2asympt}~:
\beq
d_\epsilon^{(j)}\leq\frac{n_\epsilon^{(j)}+g^{(j)}-1}{2}
\eeq
et, par le lemme~\ref{ptesdiverses}.b et \eqref{X>f2(2d-2-n)asympt}~:
\beq
|X^{(j)}(K)|>f_{2,X^{(j)}}(2d_\epsilon^{(j)}-2-n_\epsilon^{(j)}).
\eeq
On est donc en mesure d'appliquer la proposition pr\'ec\'edente, ce qui
nous fournit, pour $j$ assez grand,
un code lin\'eaire intersectant $C_\epsilon^{(j)}$
de param\`etres $[n_\epsilon^{(j)},k_\epsilon^{(j)}]$, avec
$k_\epsilon^{(j)}\geq d_\epsilon^{(j)}+1-g^{(j)}$.
On trouve alors, pour $j$ tendant vers l'infini~:
\beq
\liminf\frac{k_\epsilon^{(j)}}{n_\epsilon^{(j)}}\geq\frac{\delta_\epsilon-1}{\nu_\epsilon}=\min\left(1-\frac{5}{2\nu_\epsilon},\frac{1}{2}\frac{\nu_{2\epsilon}}{\nu_\epsilon}-\frac{1}{2\nu_\epsilon}\right)
\eeq
si $K$ est infini, tandis que si $|K|=q<+\infty$~:
\beq
\liminf\frac{k_\epsilon^{(j)}}{n_\epsilon^{(j)}}\geq\frac{\delta_\epsilon-1}{\nu_\epsilon}=\min\left(1-\frac{5}{2\nu_\epsilon}+\frac{2-\frac{2}{\nu_\epsilon}+\frac{1}{q^2-1}}{q^2-1}\,,\;\frac{1}{2}\frac{\nu_{2\epsilon}}{\nu_\epsilon}-\frac{1}{2\nu_\epsilon}\right).
\eeq
Ceci \'etant vrai pour tout $\epsilon>0$ assez petit,
on conclut par un argument diagonal.
\end{proof}

Lorsque $|K|=q<+\infty$, on note $A(q)$ le plus grand r\'eel tel
qu'il existe une suite de courbes $X^{(j)}$ de genre $g^{(j)}$ sur $K$,
avec $g^{(j)}\longto\infty$ et $\frac{|X^{(j)}(K)|}{g^{(j)}}\longto A(q)$
lorsque $j$ tend vers l'infini.

On sait que $0<A(q)\leq\sqrt{q}-1$, et que l'in\'egalit\'e de droite
est une \'egalit\'e au moins quand $q$ est un carr\'e
(\cite{DV}\cite{Ihara}\cite{TVZ}).

\begin{corollaire}
Si $A(q)\geq 4-\frac{12q^2-4}{q^4+2q^2-1}$,
on a
\beq
R_q\geq\frac{1}{2}-\frac{1}{2A(q)}.
\eeq
\end{corollaire}
\begin{proof}
On applique le corollaire pr\'ec\'edent avec $\nu=A(q)$
(en remarquant que la valeur critique $\nu=4-\frac{12q^2-4}{q^4+2q^2-1}$
est celle qui rend \'egaux les deux termes dans le $\min$ dans
\eqref{Rqnu}).
\end{proof}

Ceci am\'eliore le th\'eor\`eme~2 de \cite{21sep}, qui parvient \`a
la m\^eme conclusion, mais sous l'hypoth\`ese plus restrictive $A(q)>8$.

On a essay\'e ici de tenir compte syst\'ematiquement de la finitude
de $K$ lorsque cela permettait d'affiner les estimations.
On aurait pu s'en passer, ce qui aurait beaucoup simplifi\'e les calculs,
mais alors on aurait prouv\'e le corollaire seulement sous
l'hypoth\`ese $A(q)\geq4$. Il n'est pas clair que ceci soit
restrictif~: existe-t-il un $q$ tel que
$4-\frac{12q^2-4}{q^4+2q^2-1}\leq A(q)<4$~?
Si c'est le cas,
il serait amusant d'en exhiber un.

D'un autre c\^ot\'e, si $K$ est infini, la th\'eorie
d\'evelopp\'ee ici est essentiellement superflue. 
L'ensemble $\PP^1(K)$ \'etant infini,
on montre facilement
$R_K\geq\frac{1}{2}$ (et m\^eme, en fait, $R_K=\frac{1}{2}$)
sans avoir besoin d'utiliser de courbes de genre sup\'erieur.

\section{Application \`a la complexit\'e bilin\'eaire de la multiplication~: la borne de Shparlinski-Tsfasman-Vladut}

On commence par faire quelques rappels d'ordre g\'en\'eral sur
la complexit\'e bilin\'eaire de la multiplication
dans une extension de corps.
Pour un survol plus complet de ces questions, on pourra consulter
\cite{BR}.

Fixons un corps $K$.
Si $E_1,\dots,E_s$ sont des $K$-espaces vectoriels de dimension finie,
on d\'efinit le \emph{rang tensoriel} d'un \'el\'ement
$t\in E_1\tens\cdots\tens E_s$ comme la longueur minimale d'une
d\'ecomposition de $t$ en somme de tenseurs \'el\'ementaires.

Si $b:E_1\times E_2\longto E_3$ est une application $K$-bilin\'eaire,
on d\'efinit la \emph{complexit\'e bilin\'eaire} de $b$ comme le rang
tensoriel de l'\'el\'ement $\widetilde{b}\in E_1^\vee\tens E_2^\vee\tens E_3$
naturellement d\'eduit de $b$.

Si $L$ est une extension finie de $K$, on notera $\mu_K(L)$ la complexit\'e
bilin\'eaire
de l'application de multiplication $L\times L\longto L$,
consid\'er\'ee comme une application $K$-bilin\'eaire.
Ainsi, plus concr\`etement, $\mu_K(L)$ est le plus petit entier $n$
tel qu'il existe des formes lin\'eaires
$\phi_1,\dots,\phi_n$ et $\psi_1,\dots,\psi_n:L\longto K$,
et des \'el\'ements $w_1,\dots,w_n\in L$, tels que pour tous
$x,y\in L$ on ait
\beq
\label{algomult}
xy=\phi_1(x)\psi_1(y)w_1\,+\,\cdots\,+\,\phi_n(x)\psi_n(y)w_n.
\eeq
Plus g\'en\'eralement, on appelle \emph{algorithme de multiplication}
de longueur $n$ pour $L/K$ la donn\'ee d'une d\'ecomposition du type
\eqref{algomult}. Un tel algorithme est dit \emph{sym\'etrique}
si $\phi_i=\psi_i$ pour tout $i$.

Un algorithme de multiplication \eqref{algomult} \'etant donn\'e,
on peut, suivant \cite{FZ}\cite{LW}\cite{Winograd},
lui associer deux codes lin\'eaires
$C_\phi$ et $C_\psi\subset K^n$, images des applications d'\'evaluation
\beq
\begin{array}{cccc}
\phi: & L & \longto & \!\!\!K^n\\
& x & \mapsto & \!\!\!(\phi_1(x),\dots,\phi_n(x))
\end{array}
\quad\textrm{et}\quad
\begin{array}{cccc}
\psi: & L & \longto & \!\!\!K^n\\
& y & \mapsto & \!\!\!(\psi_1(y),\dots,\psi_n(y))
\end{array}
\eeq
respectivement.
Le fait que $L$ est une $K$-alg\`ebre \emph{int\`egre} implique,
d'une part, que ces
applications sont injectives, donc $C_\phi$ et $C_\psi$ ont dimension
$k=[L:K]$, et d'autre part, que pour tous
$c\in C_\phi$ et $c'\in C_\psi$ non nuls, les supports
de $c$ et $c'$ s'intersectent
(en particulier, si l'algorithme est sym\'etrique,
$C_\phi=C_\psi$ est un \emph{code lin\'eaire intersectant} comme d\'efini dans
la section pr\'ec\'edente).
Ainsi, toute borne sup\'erieure sur le rendement de telles paires
de codes se traduit en une minoration de $\mu_K(L)$.
Par exemple, on montre facilement par
un argument de projection que
$C_\phi$ et $C_\psi$ ont distance minimale au moins $k$,
ce qui implique alors $n\geq 2k-1$ (borne de Singleton). Ainsi,
\beq
\label{Singleton}
\mu_K(L)\geq 2[L:K]-1.
\eeq
Inversement, la m\'ethode de Chudnovsky et Chudnovsky
(\cite{ChCh}, voir aussi ci-dessous)
permet d'obtenir une majoration de $\mu_K(L)$ lin\'eaire en $[L:K]$.
Ceci motive l'introduction de constantes
\beq
m_K=\liminf_{[L:K]\to\infty}\frac{\mu_K(L)}{[L:K]}
\qquad\textrm{et}\qquad
M_K=\limsup_{[L:K]\to\infty}\frac{\mu_K(L)}{[L:K]}.
\eeq
Le cas le plus int\'eressant est celui o\`u $K=\F_q$ est un corps fini.
On notera alors
\beq
\mu_q(k)=\mu_{\F_q}(\F_{q^k}),
\eeq
et
\beq
m_q=m_{\F_q}=\liminf_{k\to\infty}\frac{\mu_q(k)}{k}
\quad\textrm{et}\quad
M_q=M_{\F_q}=\limsup_{k\to\infty}\frac{\mu_q(k)}{k}.
\eeq

Conservant les notations de la section pr\'ec\'edente, et notamment
celles des applications d'\'evaluation g\'en\'eralis\'ees $\ev_{G,D}$
et des codes $C(G,D)$, on peut maintenant \'enoncer~:

\begin{theoreme}[Chudnovsky-Chudnovsky, \cite{ChCh}]
\label{thChCh}
Soit $X$ une courbe sur un corps $K$. On note $G$ la donn\'ee de
points $P_1,\dots,P_n\in X(K)$ deux \`a deux distincts
(et d'une uniformisante en chaque $P_i$). On suppose qu'il existe un 
diviseur $K$-rationnel $D$ sur $X$, et un point ferm\'e $Q$ de $X$
hors du support de $D$,
de corps r\'esiduel $L=K(Q)$, tels que~:
\begin{enumerate}[a)]
\item l'application d'\'evaluation
$\ev_{Q,D}:\cL(D)\longto L$ est surjective
\item l'application d'\'evaluation
$\ev_{G,2D}:\cL(2D)\longto K^n$ est injective.
\end{enumerate}
Alors il existe un algorithme de multiplication sym\'etrique
de longueur $n$ pour
$L/K$, et tel que le code intersectant $C_\phi=C_\psi$ associ\'e
soit un sous-code de $C(G,D)$.
En particulier, $\mu_K(L)\leq n$.
\end{theoreme}
\begin{proof}
Choisissons une section $\sigma:L\longto\cL(D)$ de $\ev_{Q,D}$, posons
\beq
\label{phi=evosigma}
\phi=\ev_{G,D}\circ\sigma:L\longto K^n,
\eeq
et notons $\phi_i$
les composantes de
$\phi$ (et $\psi_i=\phi_i$).
Par ailleurs, via b), \'etendons l'application d'\'evaluation
$\ev_{Q,2D}:\cL(2D)\longto L$ en une application $K$-lin\'eaire
$w:K^n\longto L$, et notons $w_1,\dots,w_n\in L$ l'image par $w$ des
\'el\'ements de la base canonique de $K^n$. 
Alors, la commutativit\'e du diagramme \eqref{morphisme-algebre}, avec $D'=D$,
appliqu\'ee d'une part pour l'\'evaluation en $G$, et d'autre part
pour l'\'evaluation en $Q$,
montre que \eqref{algomult} est bien v\'erifi\'ee~:
plus pr\'ecis\'ement, si $x,y\in L$ et $f=\sigma(x)$, $g=\sigma(y)$,
\beq
\begin{split}
\phi_1(x)\phi_1(y)w_1+\cdots+\phi_n(x)\phi_n(y)w_n&=w(\ev_{G,D}(f)\ev_{G,D}(g))\\
&=w(\ev_{G,2D}(fg))\\
&=\ev_{Q,2D}(fg)\\
&=\ev_{Q,D}(f)\ev_{Q,D}(g)=xy.
\end{split}
\eeq
Enfin on a bien $C_\phi\subset C(G,D)$ par \eqref{phi=evosigma}.
\end{proof}

\begin{proposition}
\label{criteremu}
Soit $X$ une courbe de genre $g$ sur un corps $K$, munie
de points $K$-rationnels $P_1,\dots,P_n\in X(K)$
deux \`a deux distincts, et
d'un point ferm\'e $Q$ de corps r\'esiduel $L=K(Q)$
de degr\'e $k=[L:K]$.
Soit aussi $D$ un diviseur $K$-rationnel sur $X$
de degr\'e $d=\deg D$, de support
ne contenant pas $Q$, et tel que~:
\begin{enumerate}[a)]
\item $2D-G$ et $D-Q$ sont ordinaires
\item $2d-n\leq g-1\leq d-k$
\end{enumerate}
o\`u $G=P_1+\cdots+P_n$ (avec notre abus de notation habituel).
Alors $X,G,D,Q$ v\'erifient les hypoth\`eses du th\'eor\`eme~\ref{thChCh},
de sorte que $\mu_K(L)\leq n$.
\end{proposition}
\begin{proof}
Les hypoth\`eses faites sur $2D-G$ et $D-Q$ signifient pr\'ecis\'ement~:
\beq
\label{l(2D-G)=0}
l(2D-G)=0
\eeq
et
\beq
\label{l(D-Q)=deg(D-Q)+1-g}
l(D-Q)=\deg(D-Q)+1-g.
\eeq
Alors $D\geq D-Q$ et \eqref{l(D-Q)=deg(D-Q)+1-g}
impliquent $l(D)=\deg(D)+1-g$.
Par ailleurs, $\ker\ev_{Q,D}=\cL(D-Q)$, de sorte que
\beq
\dim\im\ev_{Q,D}=l(D)-l(D-Q)=\deg(Q)=k,
\eeq
donc $\ev_{Q,D}$ est bien surjective.
Enfin, $\ker\ev_{G,2D}=\cL(2D-G)=\{0\}$ par \eqref{l(2D-G)=0}, donc $\ev_{G,2D}$
est bien injective.
\end{proof}

\begin{corollaire}
\label{criterefin}
Soit $X$ une courbe de genre $g\geq1$ sur un corps $K$ parfait,
munie d'un point ferm\'e $Q$ de corps r\'esiduel $L=K(Q)$ de degr\'e
$k=[L:K]>1$,
et d'un ensemble de points rationnels $\cS\subset X(K)$.
Soit $n$ un entier naturel tel que
\beq
\label{k+g-1<d<(n+g-1)/2}
2k+g-1\leq n\leq |X(K)|
\eeq
et
\beq
\label{S>max(1,2)}
|\cS|>g+f_{2,X}(2k+2g-4-n).
\eeq
Alors il existe 
des diviseurs 
$D$ de degr\'e $d=k+g-1$, \`a support dans $\cS$,
et $G$ de degr\'e $n$ sur $X$, tels que $X,G,D,Q$
v\'erifient les hypoth\`eses du th\'eor\`eme~\ref{thChCh}.
En particulier, on a $\mu_K(L)\leq n$.
\end{corollaire}
\begin{proof}
Choisissons $P_1,\dots,P_n\in X(K)$ deux \`a deux distincts,
avec $P_1\in\cS$,
et appliquons la proposition~\ref{prop-constr}
avec
\begin{itemize}
\item $r=2$
\item $s_1=1$
\item $T_1=Q$, \ $t_1=k$
\item $s_2=2$
\item $T_2=G=P_1+\cdots+P_n$, \ $t_2=n$
\item $D_0=d_0P_1$, \ $d_0=\min(k-1,\lfloor\frac{n-1}{2}\rfloor)$
\item $d=k+g-1$.
\end{itemize}
Pour tout entier $d'$ posons
\beq
f(d')=f_{1,X}(d'-k)+f_{2,X}(2d'-n).
\eeq
La proposition~\ref{prop-constr} s'applique
d\`es lors que, pour tout $d'$ tel que $d_0\leq d'<d$, on a
\beq
\label{S>f(d')}
|\cS|>f(d').
\eeq
Or par \eqref{S>max(1,2)} on a
$|\cS|>g+f_{2,X}(2k+2g-4-n)=f(d-1)$, donc par le lemme~\ref{ptesdiverses}.a,
\eqref{S>f(d')} est bien vraie pour tout $d'<d$.

La proposition~\ref{prop-constr} fournit alors $D$ de degr\'e $d$,
\`a support dans $\cS$, tel que $2D-G$ et $D-Q$ soient ordinaires.
Puisque $k>1$, le support de $D$ ne rencontre pas $Q$.
Enfin, par \eqref{k+g-1<d<(n+g-1)/2}, on a $2d-n\leq g-1=d-k$,
et on peut alors conclure gr\^ace \`a la proposition~\ref{criteremu}.
\end{proof}

\begin{corollaire}
Soient $K$ un corps parfait, et $\kappa>0$ et $\nu>2$ deux r\'eels.
On supposera que $\kappa$ et $\nu$ v\'erifient la condition suivante~:
\beq
\label{kappanusimple}
\kappa<
\begin{cases}
\frac{\nu-2}{2} & \textrm{si $\nu\leq 4$}\\
\nu-3 & \textrm{si $4<\nu< 5$}\\
\frac{\nu-1}{2} & \textrm{si $\nu\geq 5$.}
\end{cases}
\eeq
Pour tout entier $j\geq1$, on se donne une courbe $X^{(j)}$ de genre
$g^{(j)}$ sur $K$,
munie d'un point ferm\'e $Q^{(j)}$ de corps r\'esiduel $L^{(j)}=K(Q^{(j)})$
de degr\'e \mbox{$k^{(j)}=[L^{(j)}:K]$},
et on suppose que, lorsque $j$ tend vers l'infini,
on a $g^{(j)}\longto\infty$,
\beq
\liminf\frac{|X^{(j)}(K)|}{g^{(j)}}\geq\nu,
\eeq
et
\beq
\frac{k^{(j)}}{g^{(j)}}\longto\kappa.
\eeq
Alors, pour tout $j$ assez grand, il existe des diviseurs
$D^{(j)}$, $G^{(j)}$ sur $X^{(j)}$ tels que $X^{(j)},G^{(j)},D^{(j)},Q^{(j)}$
v\'erifient les hypoth\`eses du th\'eor\`eme~\ref{thChCh},
et fournissent un algorithme de multiplication sym\'etrique
de longueur $n^{(j)}$ pour $L/K$, avec
\beq
\frac{n^{(j)}}{g^{(j)}}\longto\nu
\eeq
quand $j$ tend vers l'infini. En particulier, on a
\beq
\limsup_{j\to\infty}\frac{1}{k^{(j)}}\mu_K(L^{(j)})\leq\frac{\nu}{\kappa}.
\eeq

Lorsque $|K|=q<+\infty$, la conclusion reste valide en rempla\c{c}ant
la condition \eqref{kappanusimple} par la condition plus faible
\beq
\label{kappanuq}
\kappa<
\begin{cases}
\frac{\nu-2}{2} & \textrm{si $\nu\leq 4-\frac{10q^2-2}{q^4+2q^2-1}$}\\
\left(1-\frac{1}{q^2}\right)^{-2}\nu-3\left(1-\frac{1}{q^2}\right)^{-1} & \textrm{si $4-\frac{10q^2-2}{q^4+2q^2-1}<\nu< 5-\frac{14q^2-4}{q^4+2q^2-1}$}\\
\frac{\nu-1}{2} & \textrm{si $\nu\geq 5-\frac{14q^2-4}{q^4+2q^2-1}$.}
\end{cases}
\eeq
\end{corollaire}
\begin{proof}
Compte tenu de la d\'efinition  \eqref{defphi2nu} de $\phi_{2,\nu}$,
on remarque que la condition \eqref{kappanusimple},
ou la condition \eqref{kappanuq} si $K$ est fini,
implique~:
\beq
\label{2kappa+1<nu}
2\kappa+1<\nu
\eeq
et
\beq
\label{nu>1+phi}
\nu>1+\phi_{2,\nu}(2\kappa+2-\nu).
\eeq
Pour tout $j$ assez grand, choisissons un entier $n^{(j)}\leq|X^{(j)}(K)|$,
en faisant en sorte qu'\`a l'infini
\beq
\frac{n^{(j)}}{g^{(j)}}\longto\nu.
\eeq
Alors pour $j$ assez grand on aura 
\beq
2k^{(j)}+g^{(j)}-1\leq n^{(j)}
\eeq
par \eqref{2kappa+1<nu},
et
\beq
|X^{(j)}(K)|>g^{(j)}+f_{2,X^{(j)}}(2k^{(j)}+2g^{(j)}-4-n^{(j)})
\eeq
par \eqref{nu>1+phi} et le lemme~\ref{ptesdiverses}.b.
On peut alors appliquer le corollaire pr\'ec\'edent
(avec $\cS=X^{(j)}(K)$), ce qui permet de conclure.
\end{proof}

Comme dans la section pr\'ec\'edente,
on note $A(q)$ le plus grand r\'eel tel
qu'il existe une suite de courbes $X^{(j)}$ de genre $g^{(j)}$ sur $\F_q$,
avec $g^{(j)}\longto\infty$ et $\frac{|X^{(j)}(K)|}{g^{(j)}}\longto A(q)$
lorsque $j$ tend vers l'infini.

On note aussi $A'(q)$ le plus grand r\'eel tel
qu'il existe une suite de courbes $X^{(j)}$ de genre $g^{(j)}$ sur $\F_q$,
avec $g^{(j)}\longto\infty$ et $\frac{|X^{(j)}(K)|}{g^{(j)}}\longto A'(q)$
lorsque $j$ tend vers l'infini,
et avec la condition suppl\'ementaire
\beq
\frac{g^{(j+1)}}{g^{(j)}}\longto 1.
\eeq

\begin{corollaire}
\label{ShpaTsaVla}
Si $A(q)\geq 5-\frac{14q^2-4}{q^4+2q^2-1}$, on a
\beq
\label{muinf}
m_q\leq 2\left(1+\frac{1}{A(q)-1}\right).
\eeq
Si $A'(q)\geq 5-\frac{14q^2-4}{q^4+2q^2-1}$, on a
\beq
\label{musup}
M_q\leq 2\left(1+\frac{1}{A'(q)-1}\right).
\eeq
En particulier, si $q\geq 49$ est un carr\'e, on a
\beq
M_q\leq 2\left(1+\frac{1}{\sqrt{q}-2}\right).
\eeq
\end{corollaire}
\begin{proof}
Consid\'erons une suite de courbes $X^{(j)}$ sur $\F_q$, de genre $g^{(j)}$
tendant vers l'infini, avec $\frac{|X^{(j)}(K)|}{g^{(j)}}\longto A(q)$.
Posons $\kappa_\epsilon=\frac{A(q)-1-\epsilon}{2}$, de sorte que
$\kappa_\epsilon>0$ si $\epsilon>0$ est choisi assez petit.
Pour tout $j$ assez grand, choisissons un entier $k^{(j)}$,
en faisant en sorte que $\frac{k^{(j)}}{g^{(j)}}\longto\kappa_\epsilon$
quand $j$ tend vers l'infini. Par la borne de Weil, si $j$ est assez
grand, $X^{(j)}$ admet au moins un point $Q^{(j)}$ de degr\'e
exactement $k^{(j)}$ (cf.~\cite{Stichtenoth} Cor.~V.2.10.c).
On peut alors appliquer le corollaire pr\'ec\'edent avec
$\kappa=\kappa_\epsilon$ et $\nu=A(q)$, ce qui donne
\beq
m_q\leq\limsup_{j\to\infty}\frac{1}{k^{(j)}}\mu_q(k^{(j)})\leq\frac{\nu}{\kappa_\epsilon}=2\left(1+\frac{1+\epsilon}{A(q)-1-\epsilon}\right),
\eeq
et puisque ceci vaut pour tout $\epsilon$, on en d\'eduit \eqref{muinf}.

Consid\'erons maintenant
une suite de courbes $X^{(j)}$ sur $\F_q$, de genre $g^{(j)}$
tendant vers l'infini, avec $\frac{|X^{(j)}(K)|}{g^{(j)}}\longto A'(q)$,
et 
\beq
\label{croissancelente}
\frac{g^{(j+1)}}{g^{(j)}}\longto 1.
\eeq
Posons $\kappa_\epsilon=\frac{A'(q)-1-\epsilon}{2}$, de sorte que
$\kappa_\epsilon>0$ si $\epsilon>0$ est choisi assez petit.
Pour tout entier $k$ assez grand, notons $j_k$ le plus petit entier
tel que
\beq
2k+(1+\epsilon)g^{(j_k)}-1\leq |X^{(j_k)}(K)|.
\eeq
Alors, gr\^ace \`a \eqref{croissancelente},
on a
\beq
\frac{k}{g^{(j_k)}}\longto \kappa_\epsilon
\eeq
quand $k$ tend vers l'infini.
De m\^eme que pr\'ec\'edemment, si $k$ est assez grand,
$X^{(j_k)}$ admet un point $Q^{(k)}$ de degr\'e exactement $k$,
et on peut appliquer le corollaire pr\'ec\'edent pour trouver
\beq
M_q=\limsup_{k\to\infty}\frac{1}{k}\mu_q(k)\leq\frac{\nu}{\kappa_\epsilon}=2\left(1+\frac{1+\epsilon}{A'(q)-1-\epsilon}\right),
\eeq
et on conclut en faisant tendre $\epsilon$ vers $0$.

La derni\`ere assertion r\'esulte du fait que, si $q$ est un carr\'e,
on a $A(q)=A'(q)=\sqrt{q}-1$ (voir \cite{STV}, ``Claim'' p.~163).
\end{proof}

On a ainsi d\'emontr\'e le r\'esultat principal de \cite{STV},
dont la preuve \'etait incompl\`ete (comme indiqu\'e dans \cite{Cascudo}).
Par ailleurs, notre m\'ethode a l'avantage d'\^etre constructive, du moins
sous l'hypoth\`ese que l'on sait construire explicitement les courbes
$X$, \emph{ainsi que les points $Q$ sur ces courbes}, qui
interviennent dans la preuve
(il serait par ailleurs int\'eressant d'\'etudier si une telle
construction ne peut pas s'obtenir par une modification simple
d'autres m\'ethodes de construction de courbes d\'ej\`a connues)~;
un tel couple $(X,Q)$ \'etant
donn\'e, les propositions~\ref{prop-constr} et~\ref{criteremu}
permettent en effet de construire un algorithme de multiplication
pour $K(Q)/K$
en temps polyn\^omial en le genre $g$ de $X$, qui approche la limite
donn\'ee dans le dernier corollaire quand $g$ tend vers l'infini.
Ceci r\'epond (de mani\`ere positive) \`a la remarque finale de \cite{STV}.

\section{\'Epilogue}

Il est possible de raffiner les r\'esultats obtenus dans la section
pr\'ec\'edente, au moins dans deux directions.

Tout d'abord, on peut vouloir non pas une borne sur $m_q$ ou $M_q$,
asymptotique, mais une borne explicite sur $\mu_q(k)$ pour toute valeur
de $k$. 
Suivant la m\'ethode de \cite{Ballet1999}\cite{Ballet2003}\cite{BC},
c'est possible d\`es lors qu'on dispose de familles de courbes dont
on sache pr\'ecis\'ement minorer le nombre de points et
majorer la croissance du genre.
Par exemple, le corollaire~\ref{ShpaTsaVla} montre que, si $q\geq49$
est un carr\'e, alors pour tout $\epsilon>0$ il existe un entier $k_\epsilon$
tel que pour tout $k\geq k_\epsilon$ on ait
\beq
\frac{1}{k}\mu_q(k)\leq2\left(1+\frac{1+\epsilon}{\sqrt{q}-2}\right).
\eeq
En \'etant plus soigneux dans la preuve, on peut esp\'erer
une estimation explicite de ce $k_\epsilon$.

Par ailleurs, on peut aussi vouloir
combiner notre construction avec des variantes
de l'algorithme de Chudnovsky-Chudnovsky par
interpolation en des points de degr\'e sup\'erieur, ou avec
multiplicit\'es (\cite{BR0}\cite{Arnaud}\cite{CO}). Ceci est
int\'eressant notamment pour obtenir des informations sur les
$\mu_q(k)$ lorsque $q$ n'est pas un carr\'e.

On indique ici quelques raffinements de ce type, sans chercher
\`a \^etre exhaustif.
Par souci de simplicit\'e, on ne cherchera pas non plus \`a exploiter
les r\'esultats des sections~\ref{section1} et~\ref{section2}
dans toute leur g\'en\'eralit\'e~; la suite de l'expos\'e reposera
uniquement sur le r\'esultat suivant, qui est une version faible
de la proposition~\ref{prop-constr}~:

\begin{proposition}
\label{prop-constr-faible}
Soient $X$ une courbe de genre $g$ sur
un corps fini $\F_q$, et $r\geq 1$ un entier.
On suppose donn\'es des entiers $s_1,\dots,s_r\in\{1,2\}$ et
des diviseurs $\F_q$-rationnels $T_1,\dots,T_r$ sur $X$.
On suppose enfin
\beq
|X(\F_q)|>\sum_{i=1}^r(s_i)^2g.
\eeq
Alors,
pour tout entier $d$,
il existe un diviseur $\F_q$-rationnel $D$ sur $X$,
de degr\'e $\deg D=d$, \`a support dans $X(\F_q)$, tel que les
$s_iD-T_i$ soient ordinaires.
\end{proposition}
\begin{proof}
On applique la proposition~\ref{prop-constr} avec $\cS=X(\F_q)$,
et le lemme~\ref{ptesdiverses}.a.
\end{proof}

Remontant en arri\`ere, on observe que ce r\'esultat
n'utilise pas le lemme~\ref{lemme2} dans toute sa g\'en\'eralit\'e,
mais uniquement \`a travers la formule~\eqref{miniPlucker2}.
Ainsi si l'on s'int\'eresse uniquement aux r\'esultats de cette
section, il est possible de simplifier consid\'erablement les
preuves et les calculs.

\begin{corollaire}
\label{constrD-Qet2D-G}
Soient $X$ une courbe de genre $g$ sur
un corps fini $\F_q$, munie de deux diviseurs $\F_q$-rationnels
$Q$ et $G$. On note $k=\deg Q$ et $n=\deg G$, et on suppose
\beq
|X(\F_q)|>5g
\eeq
et
\beq
n\geq2k+g-1.
\eeq
Alors
il existe un diviseur $\F_q$-rationnel $D$ sur $X$,
\`a support dans $X(\F_q)$, tel que 
$D-Q$ soit non-sp\'ecial de degr\'e $g-1$,
et $2D-G$ sans sections.

En particulier, si $n=2k+g-1$, alors
$D-Q$ et $2D-G$ sont non-sp\'eciaux de degr\'e $g-1$.
\end{corollaire}
\begin{proof}
On applique la proposition pr\'ec\'edente avec $r=2$,
$s_1=1$, $T_1=Q$, $s_2=2$, $T_2=G$, et $d=k+g-1$
(en remarquant que pour un diviseur de degr\'e inf\'erieur \`a $g-1$,
ordinaire \'equivaut \`a sans sections, et en degr\'e $g-1$, ordinaire
\'equivaut \`a non-sp\'ecial).
\end{proof}

Une variante de ce r\'esultat est \'enonc\'ee par Ballet
(\cite{Ballet2008}, Prop.~2.1)~; malheureusement la preuve qui en est
donn\'ee reproduit l'erreur signal\'ee
par \cite{Cascudo} dans \cite{STV}
(l'auteur veut majorer le nombre de classes de diviseurs $[D]$
telles que $2D-G$ soit sp\'ecial, par le nombre de diviseurs effectifs
de degr\'e $g-1$~; pour cela il associe \`a toute telle classe
un diviseur effectif $E\sim 2D-G$, et conclut en affirmant que cette
application $[D]\mapsto E$ est injective~; cela n'est malheureusement
pas vrai en g\'en\'eral~: si le groupe de classes contient de la
$2$-torsion, on peut trouver deux diviseurs $D$ et $D'$ non lin\'eairement
\'equivalents mais tels que $2D-G$ et $2D'-G$ le soient~;
alors si par exemple
le syst\`eme lin\'eaire correspondant est de dimension~$1$,
il n'y a qu'un seul choix possible pour $E$, et on aura
$[D]\mapsto E$ et $[D']\mapsto E$).

Cette proposition assurait l'existence de diviseurs qui jouent
un r\^ole central dans les r\'esultats de \cite{Ballet2008}~;
ceux-ci \'etaient donc compromis.
En lui substituant notre m\'ethode de construction de diviseurs,
sous la forme du corollaire~\ref{constrD-Qet2D-G},
on va pouvoir corriger cet \'etat de fait
(cela ne fonctionne cependant que si les courbes mises en jeu ont
suffisamment de points).

On commence par donner un crit\`ere num\'erique simple pour estimer
la complexit\'e bilin\'eaire (qui remplace le th\'ero\`eme 2.1.(1)
de \cite{Ballet2008}).

\begin{lemme}
Soit $X$ une courbe de genre $g$ sur un corps fini $\F_q$
avec
\beq
|X(\F_q)|>5g.
\eeq
Alors pour tout entier $k$ dans l'intervalle
\beq
\label{encadrek}
\left\lceil2\log_q\frac{2g+1}{\sqrt{q}-1}\right\rceil\:<\: k\;\leq\:\frac{|X(\F_q)|+1-g}{2}
\eeq
on a
\beq
\mu_q(k)\leq 2k+g-1.
\eeq
\end{lemme}
\begin{proof}
Par \cite{Stichtenoth} Cor.~V.2.10.c,
la premi\`ere in\'egalit\'e dans \eqref{encadrek} implique
que $X$ admet (au moins) un point $Q$
de degr\'e $k$.
D'autre part, la seconde in\'egalit\'e dans \eqref{encadrek} implique
que $X$ contient au moins $n=2k+g-1$ points $P_1,\dots,P_n$ de degr\'e $1$.
Posant $G=P_1+\cdots+P_n$, on conclut alors par
le corollaire~\ref{constrD-Qet2D-G} et la proposition~\ref{criteremu}.
\end{proof}

Par exemple (\cite{Winograd}), en prenant $X=\PP^1$, et compte tenu de
\eqref{Singleton}, on trouve~:
\beq
\mu_q(k)=2k-1\qquad\textrm{pour $k\leq\frac{q}{2}+1$.}
\eeq
En prenant pour $X$ une courbe elliptique convenable, on trouverait aussi
(\cite{Shokro})~:
\beq
\label{inegShokro}
\mu_q(k)\leq 2k\qquad\textrm{pour $k<\frac{q+e(q)+1}{2}$}
\eeq
avec $e(q)\lesssim 2\sqrt{q}$, et en particulier $e(q)=2\sqrt{q}$
si $q$ est un carr\'e.

On pourrait proc\'eder de m\^eme avec des courbes de genre $2$, $3$, etc.

Un autre point de vue, \'equivalent, est le suivant.
Pour tout entier $k$, notons
$\mathcal{X}_{q,k}$ l'ensemble des courbes (\`a isomorphisme pr\`es)
$X$ sur $\F_q$, de genre not\'e $g(X)$, v\'erifiant~:
\begin{enumerate}[a)]
\item $g(X)\leq\frac{1}{2}(q^{(k-1)/2}(q^{1/2}-1)-1)$
\item $|X(\F_q)|>5g(X)$
\item $|X(\F_q)|-g(X)\geq 2k-1$.
\end{enumerate}
Alors~:
\begin{lemme}
\label{muXqk}
Pour tout corps fini $\F_q$,
et pour tout entier $k$ tel que $\mathcal{X}_{q,k}$ soit non vide,
on a
\beq
\frac{1}{k}\mu_q(k)\leq 2+\frac{\min_{X\in\mathcal{X}_{q,k}}g(X)-1}{k}.
\eeq
\end{lemme}
\begin{proof}
C'est une reformulation du lemme pr\'ec\'edent.
\end{proof}

Par rapport \`a d'autres r\'esultats semblables dans la litt\'erature,
notre m\'ethode de construction de diviseurs autorise une condition
de d\'ependance plus faible entre $k$, $g$, et le nombre de points de
la courbe. Par exemple, le th\'eor\`eme~1.1
de \cite{Ballet1999}
(ou plus pr\'ecis\'ement son corollaire~2.1)
arrive \`a la m\^eme
conclusion, mais en rempla\c{c}ant la seconde in\'egalit\'e
de \eqref{encadrek} par $k\leq\frac{|X(\F_q)|+1-2g}{2}$, ou ce qui
revient au m\^eme, en rempla\c{c}ant la condition c) d\'efinissant
$\mathcal{X}_{q,k}$ par la condition plus forte $|X(\F_q)|-2g(X)\geq 2k-1$.
En contrepartie cependant, notre m\'ethode exige que les courbes
consid\'er\'ees aient
suffisamment de points, comme exprim\'e par la condition b) ci-dessus.


\medskip

On peut maintenant continuer en suivant les arguments \cite{Ballet2008}.
En fait, on va donner une
version l\'eg\`erement plus pr\'ecise du r\'esultat qui y \'etait
\'enonc\'e.

On consid\`ere la fonction psi de Dedekind, d\'efinie pour tout entier $N$
par
\beq
\psi(N)=N\prod_{\substack{l|N\\ \textrm{$l$ premier}}}\left(1+\frac{1}{l}\right).
\eeq

\begin{lemme}
\label{courbesmodulaires}
Soient $p$ un nombre premier et $N$ un entier premier \`a $p$.
Alors la courbe modulaire
$X_0(N)$ est lisse sur $\F_p$, de genre
\beq
\label{g0N<psi}
g_0(N)\leq\frac{\psi(N)}{12},
\eeq
et poss\`ede
\beq
\label{|X0N|>psi}
|X_0(N)(\F_{p^2})|\geq(p-1)\frac{\psi(N)}{12}
\eeq
points sur $\F_{p^2}$.
\end{lemme}
\begin{proof}
Voir \cite{TV}, \S~4.1.
\end{proof}

\begin{remarque}
En fait on peut \^etre l\'eg\`erement plus pr\'ecis dans le lemme.
La formule de Hurwitz donne une expression exacte
\beq
g_0(N)=\frac{\psi(N)}{12}+1-\frac{\nu_\infty(N)}{2}-\frac{\nu_3(N)}{3}-\frac{\nu_2(N)}{4}
\eeq
avec
\begin{itemize}
\item $\;\nu_\infty(N)=\sum_{d|N}\phi(\pgcd(d,\frac{N}{d}))=\displaystyle\prod_{l^\nu||N}\textstyle\begin{cases}2l^{\frac{\nu-1}{2}} & \textrm{si }\nu\textrm{ impair}\\ (l+1) l^{\frac{\nu}{2}-1} & \textrm{si }\nu\textrm{ pair}\end{cases}$
\item $\;\nu_3(N)=\begin{cases}\prod_{l|N}\left(1+\left(\frac{-3}{l}\right)\right) & \textrm{si }9\nmid N \\ 0 & \textrm{si }9\,|\,N\end{cases}$
\item $\;\nu_2(N)=\begin{cases}\prod_{l|N}\left(1+\left(\frac{-1}{l}\right)\right) & \textrm{si }4\nmid N \\ 0 & \textrm{si }4\,|\,N\end{cases}$
\end{itemize}
tandis que la relation d'Eichler-Shimura donne
\beq
|X_0(N)(\F_{p^2})|=p^2+1+pg_0(N)-\tr T_{p^2}
\eeq
o\`u l'op\'erateur de Hecke $T_{p^2}$ agit sur l'espace des formes
paraboliques $S_2(\Gamma_0(N))$, sa trace pouvant se calculer
explicitement, par exemple par la formule donn\'ee dans \cite{Miyake},
Th.~6.8.4 et Rem.~6.8.1, pp.~263--264~:
\beq
\tr T_{p^2}=\frac{\psi(N)}{12}+\delta(N,p^2)-\sum_ta(t)\sum_fb(t,f)c(t,f)
\eeq
o\`u $\delta(N,p^2)=p^2+p+1$ si $N>1$.
Les termes $a(t)\sum_fb(t,f)c(t,f)$ sont positifs, et leur contribution
\`a la somme peut s'exprimer simplement pour certaines valeurs
particuli\`eres de $t$~:
\begin{itemize}
\item $\;\frac{1}{2}\,p\,\nu_\infty(N)$ pour $t=\pm 2p$
\item $\;\frac{1}{3}\left(p+1-\left(\frac{-3}{p}\right)\right)\nu_3(N)$ pour $t=\pm p\;$ (si $3\nmid N$)
\item $\;\frac{1}{4}\left(p+1-\left(\frac{-1}{p}\right)\right)\nu_2(N)$ pour $t=0\;$ (si $2\nmid N$)
\end{itemize}
de sorte que pour $N>1$ premier \`a $6p$~:
\beq
\begin{split}
|X_0(N)(\F_{p^2})|=(p-1)\frac{\psi(N)}{12}\,+\,\frac{1\!-\!\left(\frac{-3}{p}\right)}{3}&\nu_3(N)+\frac{1\!-\!\left(\frac{-1}{p}\right)}{4}\nu_2(N)\\
&+\!\sum_{t\neq0,\pm p,\pm 2p}^{\phantom{.}}\!\!\!a(t)\sum_fb(t,f)c(t,f).
\end{split}
\eeq
On obtiendrait des formules analogues pour $2|N$ ou $3|N$.
\end{remarque}

Pour toute partie infinie $\mathcal{A}$ de $\N$ et pour tout r\'eel $x>0$
on note
\beq
\lceil x\rceil_{\mathcal{A}}=\min\:\mathcal{A}\cap[x,+\infty[
\eeq
le plus petit \'el\'ement de $\mathcal{A}$ sup\'erieur ou \'egal \`a $x$,
et on pose
\beq
\epsilon_{\mathcal{A}}(x)=\sup_{y\geq x}\:\frac{\lceil y\rceil_{\mathcal{A}}-y}{y},
\eeq
de sorte que la fonction $\epsilon_{\mathcal{A}}$ est d\'ecroissante,
et que pour tout $x>0$,
l'intervalle $[x,(1+\epsilon_{\mathcal{A}}(x))x]$ contient au moins
un \'el\'ement de $\mathcal{A}$.


Ainsi, si $p$ est un nombre premier,
$\lceil x\rceil_{\psi(\N\setminus p\N)}$ est le plus petit $n\geq x$
qui puisse s'\'ecrire sous la forme $n=\psi(N)$
pour un entier $N$ premier \`a $p$, et ~:

\begin{lemme}
\label{psiBertrand}
Avec ces notations, pour $p\neq2$, on a
\beq
\lceil x\rceil_{\psi(\N\setminus p\N)}\leq 2x\quad\textrm{pour tout $x\geq\frac{3}{2}$,}
\eeq
ou autrement dit~:
\beq
\epsilon_{\psi(\N\setminus p\N)}(3/2)\leq 1.
\eeq
\end{lemme}
\smallskip
\begin{proof}
Si
$j=\lfloor\frac{\log 2x/3}{\log 2}\rfloor$,
on a bien $x<3\cdot2^j=\psi(2^{j+1})\leq2x$.
\end{proof}

\begin{proposition}
\label{Ballet+}
Soit $p\geq 7$ un nombre premier.
Alors pour tout $k>\frac{p^2+p+1}{2}$ on a
\beq
\frac{1}{k}\mu_{p^2}(k)\leq 2+\frac{\frac{1}{12}\left\lceil\frac{24k-12}{p-2}\right\rceil_{\psi(\N\setminus p\N)}-1}{k}.
\eeq
\end{proposition}
\begin{proof}
On choisit $N$ premier \`a $p$ tel que
$\psi(N)=\left\lceil\frac{24k-12}{p-2}\right\rceil_{\psi(\N\setminus p\N)}$
et on pose $X=X_0(N)$.
Alors par \eqref{g0N<psi} et
\eqref{|X0N|>psi} on a $|X(\F_{p^2})|-g(X)\geq(p-2)\frac{\psi(N)}{12}$
donc la condition c) pr\'ec\'edant le lemme \ref{muXqk}
est bien v\'erifi\'ee.

De m\^eme $|X(\F_{p^2})|-5g(X)\geq(p-6)\frac{\psi(N)}{12}$
donc pour $p\geq7$ la condition b) est bien v\'erifi\'ee elle aussi.


Enfin par le lemme~\ref{psiBertrand} on a
$\psi(N)=\left\lceil\frac{24k-12}{p-2}\right\rceil_{\psi(\N\setminus p\N)}\leq\frac{48k-24}{p-2}$
donc
\beq
g(X)\leq\frac{\psi(N)}{12}\leq\frac{4k-2}{p-2},
\eeq
et pour $p\geq7$ et $k>\frac{p^2+p+1}{2}$, on montre
facilement que cette derni\`ere quantit\'e
est major\'ee par $\frac{1}{2}(p^{k-1}(p-1)-1)$.
Ainsi la condition a) est v\'erifi\'ee,
et on peut conclure gr\^ace au lemme \ref{muXqk}.
\end{proof}

\begin{remarque}
Gr\^ace \`a cette proposition, toute majoration (effective)
de la fonction $\lceil .\rceil_{\psi(\N\setminus p\N)}$,
ou de $\epsilon_{\psi(\N\setminus p\N)}$,
se traduit en une majoration (effective)
des $\mu_{p^2}(k)$.
Il s'agit donc, pour un r\'eel $x>0$ donn\'e, de trouver un entier $N$
premier \`a $p$ tel que $\psi(N)$ soit sup\'erieur \`a $x$ mais aussi
petit que possible. Une premi\`ere analyse fournit deux approches
naturelles \`a ce probl\`eme.

Tout d'abord, on peut chercher $N$ parmi les entiers n'ayant
que des petits facteurs premiers. En effet, soit
$\mathcal{B}=\{l_1,\dots,l_B\}$
un ensemble de nombres premiers,
$p\not\in\mathcal{B}$. Posons $N_{\mathcal{B}}=\prod_{i=1}^Bl_i$
et supposons $\psi(N_{\mathcal{B}})=\prod_{i=1}^B(l_i+1)<x$.
Alors si $N=N'N_{\mathcal{B}}$ o\`u $N'$ a tous ses facteurs premiers
dans $\mathcal{B}$, on a $\psi(N)=N'\psi(N_{\mathcal{B}})$.
Ainsi, si l'on sait trouver un entier $N'\geq\frac{x}{\psi(N_{\mathcal{B}})}$
aussi petit que possible
ayant tous ses facteurs premiers dans $\mathcal{B}$,
on en d\'eduit une majoration de $\lceil x\rceil_{\psi(\N\setminus p\N)}$.
Pour $\mathcal{B}=\{2\}$ on retrouve le lemme~\ref{psiBertrand}.
Il serait int\'eressant d'optimiser le choix de $\mathcal{B}$
(d\'ependant \'eventuellement de $x$) pour obtenir de meilleures
estimations.

\`A l'oppos\'e, on peut aussi chercher $N$ parmi les entiers n'ayant
que des grands facteurs premiers. En effet, si $N$
n'a aucun facteur premier inf\'erieur \`a $N^{1/u}$,
alors $\psi(N)\leq N\left(1+\frac{1}{N^{1/u}}\right)^u$,
et si on sait trouver un tel $N\geq x$ aussi petit que possible,
on peut esp\'erer, en optimisant pr\'ealablement $u$, obtenir une majoration
approchant suffisamment $\lceil x\rceil_{\psi(\N\setminus p\N)}$.
Le cas extr\^eme est celui o\`u l'on prend $u=1$, c'est-\`a-dire
o\`u l'on cherche $N$ premier.
On obtient alors la majoration
\beq
\label{majpremier}
\lceil x\rceil_{\psi(\N\setminus p\N)}\leq\lceil x-1\rceil_{\mathcal{P}}+1\quad\textrm{pour $x>p+1$}
\eeq
o\`u $\mathcal{P}$ est l'ensemble des nombres premiers
(en effet, $N=\lceil x-1\rceil_{\mathcal{P}}$ est un nombre premier strictement
plus grand que $p$, et $\psi(N)=N+1\geq x$).
Ceci permet de mettre \`a profit tous les r\'esultats
connus sur la fonction $\epsilon_{\mathcal{P}}$~;
par exemple le postulat de Bertrand, prouv\'e par Tchebychev,
donne $\epsilon_{\mathcal{P}}(1)=1$, et
combin\'e avec \eqref{majpremier}, il fournit une majoration
essentiellement \'equivalente \`a celle du lemme~\ref{psiBertrand}. 
D'autres estimations plus fines sont connues, et on les utilisera
dans le corollaire ci-dessous.
Il serait int\'eressant cependant, l\`a aussi, d'\'etudier
si un choix convenable de $u>2$, d\'ependant \'eventuellement de $x$,
permet de faire signicativement mieux.
\end{remarque}

\begin{corollaire}
Soit $p\geq 7$ un nombre premier.
Alors
\begin{enumerate}[(i)]
\item pour tout $k>\frac{p^2+p+1}{2}$,
\beq
\frac{1}{k}\mu_{p^2}(k)\leq 2\left(1+\frac{1+\epsilon_{\mathcal{P}}\!\left(\frac{24k}{p-2}\right)}{p-2}\right)
\eeq
\item pour tout $k\geq1$,
\beq
\frac{1}{k}\mu_{p^2}(k)\leq 2\left(1+\frac{2}{p-2}\right)
\eeq
\item pour tout $k\geq1$,
\beq
\frac{1}{k}\mu_{p^2}(k)\leq 2\left(1+\frac{1+\frac{10}{139}}{p-2}\right)
\eeq
\item pour tout $k\geq e^{50}p$,
\beq
\frac{1}{k}\mu_{p^2}(k)\leq 2\left(1+\frac{1,000\,000\,005}{p-2}\right)
\eeq
\item pour tout $k\geq 16\,531\,(p-2)$,
\beq
\frac{1}{k}\mu_{p^2}(k)\leq 2\left(1+\frac{1+\frac{1}{25\log^2\frac{24k}{p-2}}}{p-2}\right)
\eeq
\item pour tout $k$ assez grand,
\beq
\frac{1}{k}\mu_{p^2}(k)\leq 2\left(1+\frac{1+\frac{1}{k^{0,475}}}{p-2}\right).
\eeq
\end{enumerate}
\end{corollaire}
\begin{proof}
Le point (i) d\'ecoule de la proposition~\ref{Ballet+},
de \eqref{majpremier},
et de l'in\'egalit\'e \'evidente
$\left\lceil\frac{24k-12}{p-2}-1\right\rceil_{\mathcal{P}}\leq\left\lceil\frac{24k}{p-2}\right\rceil_{\mathcal{P}}$.

Le point (ii) d\'ecoule de \eqref{inegShokro},
de la proposition~\ref{Ballet+}, et du lemme~\ref{psiBertrand}.

En remarquant que pour $p\geq 7$ et
$k>\frac{p^2+p+1}{2}$ on a $\frac{24k}{p-2}>139$,
le point (iii) d\'ecoule de \eqref{inegShokro}, de (i) et de
l'\'egalit\'e $\epsilon_{\mathcal{P}}(139)=10/139$.
Pour montrer cette derni\`ere, on remarque que si
$p_1=2$, $p_2=3$, $p_3=5$, $p_4=7$, $p_5=11$, \dots {}
est la suite des nombres premiers,
alors pour tous $n\leq n'$ on a
\beq
\epsilon_{\mathcal{P}}(p_n)=\max\left(\epsilon_{\mathcal{P}}(p_{n'}),\,\max_{n\leq j<n'}\frac{p_{j+1}-p_j}{p_j}\right).
\eeq
On pose $p_n=139$, on majore $\epsilon_{\mathcal{P}}(p_{n'})$ pour
$p_{n'}=2\,010\,881$ au moyen de \cite{Schoenfeld}
(ou bien pour $p_{n'}=396\,833$ au moyen de \cite{Dusart})
et on conclut en calculant explicitement les derniers termes
pour $n\leq j<n'$.

Enfin les points (iv), (v) et (vi) se d\'eduisent de (i)
et des estimations de 
\cite{RS}, \cite{Dusart}, et \cite{BHP}, respectivement. 
\end{proof}

En utilisant d'autres familles de courbes
(par exemple celles utilis\'ees dans \cite{STV}),
on pourrait obtenir des r\'esultats semblables pour $q=p^{2m}$
avec $m$ quelconque.
Ce cas est \'etudi\'e aussi dans \cite{Ballet2008}, cependant,
signalons encore une petite impr\'ecision dans la preuve~:
l'auteur applique le postulat de Bertrand sans prendre garde qu'il
n'a pas affaire \`a l'ensemble de tous les nombres premiers,
mais seulement \`a ceux qui
se scindent compl\`etement dans une certaine extension ab\'elienne $L$
de $\Q$. Sur le principe, cette strat\'egie de preuve reste correcte,
mais il faut pour cela
substituer au postulat de Bertrand une estimation sur la taille des
intervalles contenant des nombres premiers dans une suite arithm\'etique
donn\'ee.

\medskip

On peut aussi obtenir des estimations pour $q$ quelconque
(non n\'ecessairement carr\'e), cependant l'utilisation de la
m\'ethode de Chudnovsky-Chudnovsky telle que pr\'esent\'ee dans
le th\'eor\`eme~\ref{thChCh}, o\`u l'on \'evalue en des points
simples de degr\'e $1$, semble ne pas conduire aux meilleurs
r\'esultats.
Diverses variantes de la m\'ethode rel\^achant cette derni\`ere
condition ont \'et\'e introduites (\cite{BR0}\cite{Arnaud}), 
la plus g\'en\'erale \'etant celle de \cite{CO}, qui autorise
un diviseur d'\'evaluation $G$ quelconque.

On va voir ci-dessous comment notre m\'ethode de construction
de diviseurs permet de pr\'eciser certains r\'esultats de \cite{CO}.
On reprend les notations qui y sont introduites, et
en particulier, pour tout corps fini $\F_q$,
on note $\widehat{M}_q(u)$ la complexit\'e bilin\'eaire
(sur $\F_q$)
de la multiplication dans l'anneau local $\F_q[[t]]/t^u$.

\begin{proposition}
\label{CO+}
Soient $X$ une courbe de genre $g$ sur
un corps fini $\F_q$, et $k>1$ un entier naturel.
Soit aussi
\beq
G=u_1P_1+\cdots+u_NP_N
\eeq
un diviseur effectif sur $X$, o\`u les $P_i$ sont des points ferm\'es
de degr\'es arbitraires.
On suppose que $X$ admet un point ferm\'e $Q$ de degr\'e $k$
(c'est le cas par exemple si $2g+1\leq q^{(k-1)/2}(q^{1/2}-1)$),
et que
\beq
|X(\F_q)|>5g
\eeq
et
\beq
\deg G\geq2k+g-1.
\eeq
Alors
\beq
\mu_q(k)\leq\sum_{i=1}^N\mu_q(\deg(P_i))\widehat{M}_{q^{\deg(P_i)}}(u_i).
\eeq
\end{proposition}
\begin{proof}
Le corollaire~\ref{constrD-Qet2D-G} assure l'existence d'un diviseur
$D$ tel que $D-Q$ soit non-sp\'ecial de degr\'e $g-1$,
et $2D-G$ sans sections.
Quitte \`a remplacer $D$ par un diviseur lin\'eairement \'equivalent
on peut, par le th\'eor\`eme d'approximation forte, supposer que $D$
est de support disjoint de $Q$ et $G$ (on pourrait aussi modifier
les \'enonc\'es de \cite{CO} par l'introduction d'applications
d'\'evaluation g\'en\'eralis\'ees comme dans notre
d\'efinition~\ref{defevgen}).
En reprenant les notations de \cite{CO},
la condition que $D-Q$ est non-sp\'ecial implique que
l'application d'\'evaluation $\cL(D)\longto\cO_Q/Q$ est surjective,
tandis que celle que $2D-G$ est sans sections implique que
\beq
\phi:\cL(2D)\longto \cO_{P_1}/P_1^{\,u_1}\times\cdots\times\cO_{P_N}/P_N^{\:u_N}
\eeq
est injective, ce qui permet d'appliquer
le th\'eor\`eme~3.1 de \cite{CO} et de conclure.
\end{proof}

On remarquera que notre proposition am\'eliore le th\'eor\`eme~3.2
de \cite{CO}, qui demande $\deg G\geq 2k+2g-1$.
De fa\c{c}on \'equivalente, notre proposition donne un crit\`ere num\'erique
simple pour assurer la validit\'e des hypoth\`eses du th\'eor\`eme~3.6
de \cite{CO}, et plus particuli\`erement l'injectivit\'e de
l'application $\phi$ qui y est demand\'ee
(en notant par ailleurs une impr\'ecision dans
l'\'enonc\'e de ce th\'eor\`eme~3.6~:
il ne suffit pas de demander que $D$ soit non-sp\'ecial,
c'est $D-Q$ qui doit l'\^etre).

Cependant, comme pr\'ec\'edemment, notre m\'ethode exige
en contrepartie que la courbe
ait suffisamment de points de degr\'e $1$, par la condition
$|X(\F_q)|>5g$. On pourrait rel\^acher tr\`es l\'eg\`erement cette condition,
en utilisant la proposition~\ref{prop-constr} dans toute sa g\'en\'eralit\'e,
plut\^ot que la proposition~\ref{prop-constr-faible} qui en est un cas
particulier.
Ceci n'apporterait toutefois que peu de marge de man{\oe}uvre
suppl\'ementaire, il resterait une condition de cardinalit\'e
sur le nombre de points de la courbe.
S'affranchir d'une telle condition, avec notamment pour objectif
la possibilit\'e de couvrir le cas o\`u $q$ est petit, semble encore
un probl\`eme ouvert dans notre approche.

Remarquons enfin que, m\^eme pour $q$ grand, cette condition
$|X(\F_q)|>5g$ n'est
pas anodine dans la recherche d'une bonne famille de courbes~;
en effet, par l'interm\'ediaire de la borne de Drinfeld-Vladut
g\'en\'eralis\'ee (\cite{Serre}\cite{Tsfasman}), elle impose aussi des
contraintes sur le nombre de points de degr\'e sup\'erieur.
Illustrons ceci par un exemple.
Pour toute courbe $X$ sur $\F_q$, notons $N_m(X/\F_q)$ le nombre de points
ferm\'es de degr\'e $m$ de $X$, de sorte que pour tout $n$ on a
\beq
|X(\F_{q^n})|=\sum_{m|n}mN_m(X/\F_q).
\eeq
Si dans la proposition~\ref{CO+} on se restreint \`a ne consid\'erer
que des points de degr\'e $1$ ou $2$, et sans multiplicit\'es ($u_i=1$),
on trouve (voir aussi \cite{BR0})~:
s'il existe une courbe $X$ de genre $g$ sur $\F_q$ telle que~:
\begin{enumerate}[(i)]
\item $N_k(X/\F_q)>0\;$ (c'est le cas par exemple si $2g+1\leq q^{(k-1)/2}(q^{1/2}-1)$)
\item $N_1(X/\F_q)>5g$
\item $N_1(X/\F_q)+2N_2(X/\F_q)\geq2k+g-1$
\end{enumerate}
alors
\beq
\mu_q(k)\leq N_1(X/\F_q)+3N_2(X/\F_q).
\eeq
Or la borne de Drinfeld-Vladut g\'en\'eralis\'ee implique, pour toute
famille de courbes de genre tendant vers l'infini~:
\beq
\limsup\frac{1}{g}\left(\frac{N_1(X/\F_q)}{\sqrt{q}-1}+\frac{2N_2(X/\F_q)}{q-1}\right)\leq 1
\eeq
de sorte que la condition (ii) implique que, dans (iii), on aura au mieux~:
\beq
\limsup\frac{1}{g}\left(N_1(X/\F_q)+2N_2(X/\F_q)\right)\leq q-5\sqrt{q}-1.
\eeq
Cela interdit notamment d'utiliser des familles de courbes optimales,
pour lesquelles $\frac{1}{g}2N_2(X/\F_q)\longto q-1$.

\end{document}